\documentclass{amsart}

\usepackage[utf8]{inputenc}
\usepackage[T1]{fontenc}

\usepackage{a4wide}
\usepackage{changepage}
\usepackage[normalem]{ulem}
\usepackage{soul}
\usepackage{xcolor}
\usepackage{amsfonts,amsthm,amssymb,amsmath}
\usepackage{mathtools}
\usepackage{wasysym}
\usepackage{xspace}
\usepackage{graphicx}
\usepackage{breqn}
\usepackage{algorithm}
\usepackage{algorithmic}
\usepackage{tabularx}
\usepackage{enumerate}
\usepackage{tcolorbox}

\usepackage[english]{babel} 
\usepackage{caption}
\DeclareCaptionLabelFormat{sc-parens}{(\textsc{#2})}
\usepackage[subrefformat=sc-parens]{subcaption}
\usepackage{paralist}
\usepackage{multirow}

\usepackage{hyperref}
\hypersetup{colorlinks=true, citecolor=darkblue, linkcolor=darkblue}
\usepackage{hypcap}

\usepackage[noabbrev,capitalise]{cleveref}
\usepackage{autonum}
\usepackage{tikz}

\newtheorem{theorem}{Theorem}[section]
\newtheorem{proposition}[theorem]{Proposition}
\newtheorem{lemma}[theorem]{Lemma}

\newtheorem{claim}[theorem]{Claim}
\newtheorem{corollary}[theorem]{Corollary}
\newtheorem{definition}[theorem]{Definition}
\newtheorem{property}[theorem]{Property}

\newtheorem{fact}[theorem]{Fact}
\theoremstyle{remark}
\newtheorem{remark}{Remark}[section]
\newtheorem{convention}{Convention}
\newtheorem*{introExample*}{Introductory Example}
\newtheorem{example}{Example}

\newtheorem*{example*}{Example}

\crefname{theorem}{Theorem}{Theorems}
\crefname{lemma}{Lemma}{Lemmas}

\definecolor{darkblue}{rgb}{0,0,0.7} 

\newcommand{\pct}{\preceq}
\newcommand{\cO}{\ensuremath{\mathcal{O}}}
\newcommand{\undir}[1]{\bar{#1}}

\renewcommand{\v}{\mathrm{v}}
\newcommand{\Eup}{E^{\uparrow}\!}
\newcommand{\Edo}{E^{\downarrow}\kern-.5mm}
\newcommand{\lamZ}{\lambda^{\!0}}
\newcommand{\lamU}{\lambda^{\!1}}
\newcommand{\ptdeux}{\scriptscriptstyle{(2)}}

\DeclareMathOperator{\bc}{\mathcal{BC}}

\DeclareMathOperator{\us}{us}

\renewcommand{\u}{\ensuremath{\mathfrak{u}}}
\newcommand{\m}{\ensuremath{\mathfrak{m}}}
\renewcommand{\o}{\ensuremath{\mathfrak{o}}}
\newcommand{\s}{\ensuremath{\mathfrak{s}}}
\renewcommand{\t}{\ensuremath{\mathfrak{t}}}
\renewcommand{\r}{\ensuremath{\mathfrak{r}}}
\renewcommand{\c}{\ensuremath{\mathfrak{c}}}
\newcommand{\w}{\ensuremath{\mathrm{w}}}
\renewcommand{\l}{\ensuremath{\mathfrak{l}}}

\newcommand{\cT}{\ensuremath{\mathcal{T}}}
\newcommand{\cL}{\ensuremath{\mathcal{L}}}
\newcommand{\cR}{\ensuremath{\mathcal{R}}}
\newcommand{\yy}{\scalebox{0.6}{\LEFTcircle}}

\graphicspath{{./images/}}

\usepackage{todonotes}

\title[Combinatorial proof of the rationality of the bivariate enumeration of maps]{Combinatorial proof for the rationality of the bivariate generating series of maps in positive genus}

\author{Marie Albenque}
\address{IRIF UMR8243, Université Paris Cité, 75013 Paris, France}
\email{marie.albenque@irif.fr}
\urladdr{\url{https://www.irif.fr/~malbenque/}}

\author{Mathias Lepoutre}
\address{LIX UMR7161, \'Ecole Polytechnique, 91120 Palaiseau, France}
\email{mathias.lepoutre@lix.polytechnique.fr}
\urladdr{\url{http://www.lix.polytechnique.fr/Labo/Mathias.Lepoutre/}}

\usepackage{mathtools} 

\begin{document}

\begin{abstract}
In this paper, we give the first combinatorial proof of a rationality scheme for the generating series of maps in positive genus enumerated by both vertices and faces. This scheme was first obtained by Bender, Canfield and Richmond in 1993 by purely computational techniques. 

To do so, we rely on a bijection obtained by the second author in a previous work between those maps and a family of decorated unicellular maps. Our main contribution consists in a fine analysis of the family of unicellular maps. As a byproduct, we also obtain a new and simpler combinatorial proof of the rationality scheme for the generating series of maps enumerated by their number of edges, originally obtained computationally by Bender and Canfield in 1991 and combinatorially by the second author in 2019. 
\end{abstract}

\maketitle

\section{Introduction}
This article deals with the enumerative properties of maps of genus $g$.
Informally, a \emph{map} of genus $g$ is a proper embedding of a graph on a $g$-hole torus (precise definitions of all the terminology appearing in the introduction will be given in Section~\ref{sec:def}).
Maps of genus 0 -- also called \emph{planar maps} --  were first studied by Tutte in the 60's. In a series of papers, he developed a robust method that allowed him to obtain many enumerative results, through the resolution of generating series obtained by a combinatorial decomposition \cite{Tut62,Tut63}. 
Since then, the enumeration of maps has been studied extensively in various fields of mathematics and mathematical physics; other classical techniques to enumerate maps include matrix integrals, topological recursion and bijection with decorated trees (see for instance \cite{LanZvo04}, \cite{Eyna16}, \cite{Scha15} and references therein). 

The first enumerative results for maps on surfaces of positive genus were obtained by Lehman and Walsh~\cite{WalshLehmanI,WalshLehmanII}, who gave expressions for the generating series of maps with fixed excess (the excess is the difference between the number of edges and the number of vertices minus~1). Then, Bender and Canfield combined the method developed by Tutte together with some fine manipulation of recurrence relations on the generating functions to obtain first the asymptotic behavior of maps of a given genus enumerated by edges \cite{BenCan86}. For $g\geq 0$, denote by $\mathcal{M}_g$ the set of rooted maps of genus $g$ and denote $M_g(z)$ their following generating series: 
\begin{equation}
 M_g(z)=\sum_{\m\in \mathcal{M}_g}z^{|E(\m)|}, \quad\text{ where }E(\m)=\{\text{edges of }\m\}.
\end{equation}
Then, in~\cite{BenCan91}, they refined their result to obtain closed formulas for $M_1(z)$, $M_2(z)$ and $M_3(z)$. But, more importantly, they exhibited the following \emph{rationality scheme}: 
\begin{theorem}[Tutte~\cite{Tut63} for $g=0$, Bender and Canfield~\cite{BenCan91} for $g\geq 1$]\label{thm:BC91}
Let $T(z)$ be the unique formal power series in $z$ that satisfies $T(z)=z+3T^2(z)$ and $T(0)=0$. Then, for any $g\geq 0$, $M_g(z)$ is a rational function of $T(z)$. 
\end{theorem}

Though the number of maps with a fixed genus and a fixed number of edges is finite, this is not the case for maps with a fixed genus and a fixed number of vertices (or a fixed number of faces). But, in view of Euler's formula, which states that for any map $\m$ of genus $g$: 
\[
    |V(\m)|+|F(\m)|=2-2g+|E(\m)|, 
\]
(where $F(\m)$ and $V(\m)$ denote respectively the set of faces and of vertices of $\m$), a natural way to refine the result of Bender and Canfield is to count maps by both their number of vertices and their number of faces. Hence, for any fixed $g\geq 1$, we denote by $M_g(z_\bullet, z_\circ)$ to be the generating series of rooted maps of genus $g$ enumerated by both their vertices and faces, i.e.: 
\[
   M_g(z_\bullet, z_\circ)=\sum_{\m\in \mathcal{M}_g}z_\bullet^{|V(\m)|}z_{\circ}^{|F(\m)|}. 
\]  
Arquès obtained the first results in this direction and gave explicit formulae for $M_0$ and $M_1$ in~\cite{Arq85,Arq87}. Then, in~\cite{BeCaRi93}, Bender, Canfield and Richmond refined Theorem~\ref{thm:BC91} and obtained the following \emph{bivariate rationality scheme}:
\begin{theorem}[Arquès~\cite{Arq85,Arq87} for $g=0,1$, Bender, Canfield and Richmond \cite{BeCaRi93} for $g\geq 2$]\label{thm:bivExprRat} Define\\ $T_{\bullet}(z_\bullet, z_\circ)$ and $T_{\circ}(z_\bullet, z_\circ)$ as the unique formal power series with constant term equal to 0 and that satisfy
\begin{equation}\label{eq:defTnoirTblanc}
\begin{cases}
T_\bullet &= z_\bullet +T_\bullet^2 + 2T_\circ T_\bullet\\ T_\circ&=z_{\circ}+T_{\circ}^2+2T_\circ T_{\bullet}
\end{cases}
\end{equation}
Then \emph{$M_g(z_\bullet, z_\circ)$ {is a rational function of} $T_\bullet$ and $T_\circ$}. 
\end{theorem}

Since the formal power series $T$, $T_{\circ}$ and $T_\bullet$ appearing in Theorems~\ref{thm:BC91} and~\ref{thm:bivExprRat} can naturally be interpreted as the generating series of some decorated trees, the rationality schemes obtained in Theorems~\ref{thm:BC91} and~\ref{thm:bivExprRat} call for a combinatorial explanation. By that, we mean a decomposition of rooted maps into simple objects -- enumerated by a rational generating function -- onto which are glued the trees enumerated by either $T$ or by $T_{\circ}$ and $T_\bullet$. 
\medskip

The quest for combinatorial proofs of enumerative results obtained for maps has already quite a long history. 
The first bijection between planar maps and some trees was obtained by Cori and Vauquelin in~\cite{CorVau81}. It gave as an enumerative corollary Tutte's enumeration formula for planar maps. But, the real breakthrough came with major results obtained by Schaeffer~\cite{Scha97,Scha98} who obtained two families of bijections between planar maps and some decorated trees. Since then, those bijections have been extended and generalized and it is now customary to call the first one a \emph{blossoming bijection} and the second one a \emph{mobile-type bijection} (a term first coined in~\cite{BDG04}). In a blossoming bijection, the decorated tree associated to a map is one of its spanning trees with some decorations, which permit to reconstruct its facial cycles. In a mobile-type bijection, the tree is not necessarily spanning and its vertices carry some integral labels which encode some metric properties of the map. In the original works of Schaeffer, it was also proved that both these bijections provide a combinatorial explanation of Tutte's enumeration formula. Moreover, it is noteworthy for our purpose that the bijection obtained in~\cite{Scha97} allows to get a combinatorial proof of the bivariate enumeration of planar maps (i.e. the case $g=0$ in Theorem~\ref{thm:bivExprRat}). Let us already unveil that the key bijection in this paper is a blossoming bijection.   
\smallskip

\begin{figure}[t]
  \centering
  \subcaptionbox{A (4-valent and bicolorable) map of genus 1,\label{subfig:Genus1Intro}}{\;\includegraphics[page=5,width=0.35\linewidth]{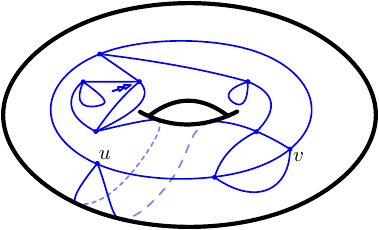}}\qquad\qquad\qquad
  \subcaptionbox{and its corresponding spanning decorated unicellular map.\label{subfig:BlossomingIntro}}{\;\includegraphics[page=6,width=0.35\linewidth]{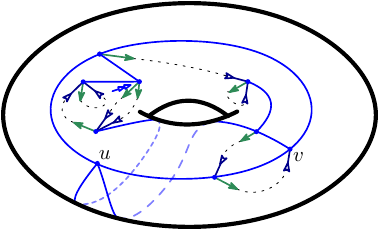}}
 \caption{Illustration of the bijection of~\cite{Lep19} for a 4-valent map on a torus of genus 1.\\ We associate to a map (represented in~\subref{subfig:Genus1Intro}) one of its spanning unicellular submaps decorated by some opening and closing half-edges (respresented in~\subref{subfig:BlossomingIntro}). These half-edges can be matched as in a parenthis word to reconstruct the facial cycles.}
  \label{fig:ToreBij}
\end{figure}

To get combinatorial proofs of Theorems~\ref{thm:BC91} and~\ref{thm:bivExprRat}, it is hence natural to generalize those bijections to higher genus. The natural analogue of trees in higher genus being unicellular maps (i.e. maps with only one face), we hence aim at bijections between maps of genus $g$ and decorated \emph{unicellular maps of genus $g$}, rather than decorated plane trees. 

In \cite{ChMaSc09}, Chapuy, Marcus and Schaeffer extended the bijection of \cite{Scha98}, and thus obtained a mobile-type bijection for maps in higher genus. But, unfortunately, it did not give a combinatorial proof of the rationality scheme obtained in~\cite{BenCan91}, but only of a slightly weaker version of it.
In~\cite{Lep19}, the second author managed to extend Schaeffer's blossoming bijection of~\cite{Scha97} to any genus, see Figure~\ref{fig:ToreBij}. And, by analyzing the decorated unicellular maps obtained via this construction, he was able to give a fully combinatorial proof of Theorem~\ref{thm:BC91}. 
\medskip

The purpose of this article is to give a \emph{combinatorial proof} of \cref{thm:bivExprRat}, based on the same bijection. Let us first mention that (as in all the works aforementioned) we do not study directly general maps. Rather, we rely on the radial construction -- introduced by Tutte~\cite{Tut63} and recalled in Section~\ref{sub:radial} -- which gives a bijection between maps and bicolorable 4-valent maps (i.e. maps where all the vertices have degree 4 and that admit a proper coloring of their faces in black and white). It turns out that via this bijection, the number of vertices and of faces of the original map are sent to the number of black and white faces of its image. 

Two main difficulties arise in the bivariate analysis of the unicellular maps obtained by applying the bijection of~\cite{Lep19} to 4-valent bicolorable maps. First, an operation on unicellular maps (a.k.a. the rerooting procedure), which is critical to enumerate them, does not behave well when applied to a single map. But, we manage to prove that its effect is somehow balanced when applied to a whole subset of maps. This is the purpose of Theorem~\ref{thm:shortcut}. As far as we are aware, this type of results is not available for mobile-type bijections and explains why the bijection obtained in~\cite{Scha98} and its generalization to higher genus given in~\cite{ChMaSc09} fail to obtain bivariate enumerative corollary.

Secondly, the method developed in~\cite{Lep19} to analyze unicellular blossoming maps cannot be extended to the bivariate enumeration. We develop here a new method, by proving some refined rationality schemes for subsets of unicellular blossoming maps. As a corollary, we obtain a new (and easier) proof of the original result of~\cite{Lep19}. Let us be slightly more specific. The main reason why the bijection of~\cite{Lep19} yields a combinatorial proof of Theorem~\ref{thm:BC91} is the following. Define informally the scheme of a unicellular map as the map obtained by removing iteratively vertices of degree 1 and 2. Then, the generating series of unicellular blossoming maps obtained by the  bijection and with a prescribed scheme are also rational functions in terms of the series $T$ defined in Theorem~\ref{thm:BC91}. It turns out that an analogous statement also holds for the bivariate generating series of unicellular maps with a prescribed scheme, which is rational in terms of the series $T_\bullet$ and $T_\circ$ defined in Theorem~\ref{thm:bivExprRat}. What we prove precisely, which is stated in Theorem~\ref{thm:MirrorR}, is an even more refined rationality scheme for unicellular maps which share the same scheme and some other properties.

Our result hence confirms that, in higher genus, blossoming bijections produce the right ``elementary objects'', and demonstrates the robustness of blossoming bijections. Another illustration of their robustness is also its recent generalization to maps on non-orientable surfaces obtained by Maciej Do{\l}\k{e}ga and the second author in~\cite{DolegaLepoutre}, which relies for some parts on an earlier version of the present work which appears in the PhD thesis of the second author \cite[Chapter II.2]{LepoutrePhD}. 
\medskip

Let us conclude this introduction on a related note about the bivariate enumeration of maps. In~\cite{BeCaRi93}, the authors obtained in fact a more precise result that Theorem~\ref{thm:bivExprRat} only and gave a general formula for $M_g$ in terms of $T_\bullet$ and $T_\circ$. This formula was precise enough to deduce the asymptotic enumeration of maps of genus $g$. Their formula (further simplified by Arquès and Giorgetti in~\cite{ArqGio99}) is given below:
\begin{theorem}[\cite{BeCaRi93,ArqGio99}]\label{thm:bivExprFormula}
For any fixed $g\geq 1$, $M_g(z_\bullet, z_\circ)$ has the following form:
\begin{equation}
M_g(z_\bullet, z_\circ)=\frac{T_{\bullet}T_{\circ}(1-T_{\bullet}-T_{\circ})P_g(T_{\bullet},T_{\circ})}{\left((1-2T_{\bullet}-2T_{\circ})^2-4T_{\bullet}T_{\circ}\right)^{5g-3}},
\end{equation}
where $T_{\bullet}$ and $T_{\circ}$ are the generating series defined in \eqref{eq:defTnoirTblanc}, and $P_g(x,y)$ is a symmetric polynomial with integer coefficients, of total degree not greater than $6g-6$.
\end{theorem}
A much more recent result about the bivariate enumeration of maps was given by Carrell and Chapuy~\cite{CarCha15}. Using the so-called KP-hierarchy, they obtain very nice recurrence relations for the number of maps with given genus and given number of vertices and faces, which (taking as an input the rationality result of Theorem~\ref{thm:BC91}) allows them to compute very efficiently the coefficients of the polynomial $P_g$ appearing in~\cref{thm:bivExprFormula}. This work has then been extended also to non-orientable maps, using this time the BKP hierarchy, by Bonzom, Chapuy and Do{\l}\k{e}ga in~\cite{BonzomChapuyDolega}.
\medskip

\textbf{Organization of the article:}
\cref{sec:def} is dedicated to introducing the terminology. We define the necessary material about bicolorable maps, the radial construction and orientations.  In \cref{sec:bijection}, we describe the closure operation, and prove in Theorem~\ref{thm:shortcut} that we can carry out the bivariate enumeration of unicellular blossoming maps even after some rerooting steps. In \cref{sec:motzkin}, we study the bivariate series of some Motzkin paths naturally appearing in \cref{sec:bijection}, so as to prove, in \cref{lem:criterion}, that the rationality stated in \cref{thm:bivExprRat} amounts to the symmetry of a certain generating function. In~\cref{sec:genSeries}, we focus on the generating series of unicellular maps refined by their scheme and some other statistics and state in Theorem~\ref{thm:MirrorR}, the refined rationality scheme that is key in our combinatorial proof of Theorem~\ref{thm:BC91}. We finally prove this theorem, first in \cref{sec:uniSym} for the univariate scheme, thus improving the proof \cite{Lep19}, and then in~\cref{sec:bivSym} for the bivariate case, concluding the bijective proof of the rationality expressed in \cref{thm:bivExprRat}.
\medskip

\textbf{Acknowledgment:}
Both authors warmly thank Éric Fusy for suggesting this problem and an anonymous referee for his/her thorough reading, which improves significantly the presentation of this work. 

Both authors were partially supported by the French ANR grant GATO~(16\,CE40\,0009), and MA by the ANR grant IsOMa (ANR-21-CE48-0007).

\section{Notations, definitions and first constructions}\label{sec:def}

\subsection{Notations}
In this article, combinatorial families are named with calligraphic letters (e.g. $\mathcal{M}$), their generating series in the corresponding capital letter (e.g. $M(z)$), and an object of the family is usually denoted by the corresponding lower case letter in gothic font (e.g. $\mathfrak{m}$). 

For $n\in \mathbb{Z}_{>0}$, we set $[n]=\{1,\ldots,n\}$, and for $i,j\in \mathbb{Z}_{\geq 0}$, with $i\leq j$, we set $[i,j]=\{i,\ldots,j\}$. 
\subsection{Maps}\label{sec:maps}

An \emph{embedded graph} is a proper embedding of a connected graph on a surface, considered up to orientation-preserving homeomorphisms. 
A \emph{face} of an embedded graph is a connected component of its complementary in the surface. 
An embedded graph is a \emph{map} if all its faces are topologically equivalent to disks, see Figure~\ref{fig:Tore}.
In this work we only consider maps on compact orientable surfaces without borders, which can be classified by their genus. The \emph{genus} of a map is hence defined as the genus of its underlying surface.

\begin{figure}[t]
  \centering
  \subcaptionbox{A graph embedded on the torus.\label{subfig:notMaps}}{\;\includegraphics[page=2,width=0.35\linewidth]{images/Tore_HDR.pdf}}\qquad\qquad\qquad
  \subcaptionbox{A (4-valent) map of genus 1.\label{subfig:MapsTore}}{\;\includegraphics[page=1,width=0.35\linewidth]{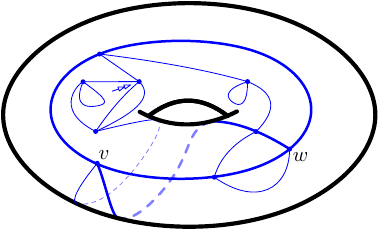}}
 \caption{Two embeddings of a connected graph on the torus of genus 1. On~\subref{subfig:notMaps}, the embedding is not cellular: the shaded face is indeed not homeomorphic to a disk. On~\subref{subfig:MapsTore}, the embedding is cellular and defines a map of genus 1.}
  \label{fig:Tore}
\end{figure}

For a map $\m$, we denote respectively $V(\m)$, $E(\m)$ and $F(\m)$, its set of vertices, edges and faces. 
\medskip

An edge can be seen as the reunion of two \emph{half-edges}, such that the removal of the middle point of an edge separates the two half-edges.  The \emph{degree} of a vertex is the number of half-edges adjacent to it. Equivalently, it is the number of edges adjacent to this vertex, with multiplicity. Similarly, the \emph{degree} of a face is the number of edges adjacent to this face, with multiplicity.

The embedding fixes the cyclical order of edges around each vertex, which defines readily a \emph{corner} as a couple of consecutive edges around a vertex. Corners may also be viewed as incidences between vertices and faces.
More precisely, a \emph{corner} is an ordered pair $(h,h')$, where $h$ and $h'$ are two half-edges adjacent to the same vertex $v$ such that $h'$ directly follows $h$ in counterclockwise order around $v$\footnote{Note that, the embedding of the map (up to orientation-preserving homeomorphisms) is completely characterized by the cyclic ordering of the edges around each vertex. This fact will come in handy in the following as it allows us to represent maps without the underlying surface.}. The set of corners of a map $\m$ is denoted $C(\m)$
. For $\kappa \in C(\m)$, we denote by $\v(\kappa)$ the vertex incident to $\kappa$.

A map $\m$ is \emph{rooted} by marking one of its corners, which is called the \emph{root corner} and denoted $\rho_\m$
 or $\rho$ if there is no ambiguity. In figures, the root corner is indicated by a double arrow, see Figure~\ref{fig:Tore}\subref{subfig:MapsTore}. The vertex and the face incident to $\rho_\m$ are respectively called the root vertex and the root face and are respectively denoted $\v(\rho)$ and $f_{\rho}$.

\subsection{Bicolorable maps and the radial construction}\label{sub:radial}
A map is called \emph{bicolorable} if its faces can be properly colored with two colors, where a coloring is said to be proper if all the edges of the map separate faces of distinct colors\footnote{The dual of bicolorable maps are bipartite maps which are more often encountered in the literature}. It is easy to see that a bicolorable map admits only two proper colorings of its faces, and that one is obtained from the other by switching the color of every face. 

\begin{convention}
In the following, all the rooted bicolorable maps are assumed to be properly colored in black and white with their root face being black, see Figure~\ref{fig:radial}\subref{subfig:radial2}. With this convention, for $\m$ a rooted bicolorable map, we denote respectively $F_\bullet(\m)$ and $F_\circ(\m)$ the set of black and white faces of $\m$.
\end{convention}
\medskip

Fix $\m$ a bicolorable map. The existence of a proper bicoloring implies immediately that for any $v\in V(\m)$, the degree of $v$ is even. Therefore, a bicolorable map is \emph{Eulerian}. In genus 0, a map is Eulerian if and only if it is bicolorable (and Schaeffer's original paper deals indeed with \emph{Eulerian planar maps}). But this equivalence fails in positive genus, where some Eulerian maps may not be bicolorable, see for instance Figure~\ref{fig:schemes}\subref{subfig:schemeCarre}. It appears that the right family of maps to consider in order to generalize Schaeffer's construction is the family of bicolorable maps rather than Eulerian maps. 
\smallskip

We denote by $\bc_g$ the set of bicolorable maps of genus $g$, and introduce their generating function $BC_g$ when enumerated by both their number of black and white faces and by their degree distribution. More precisely, we set: 
\begin{equation}\label{eq:defBC}{}
BC_g(z_\bullet,z_\circ;d_1,d_2,\ldots;t)=\sum_{\m\in \bc_g}t^{|E(\m)|}z_\bullet^{|F_\bullet(\m)|}z_\circ^{|F_\circ(\m)|} \prod_{i\geq 1} d_i^{n_i(\m)}, 
\end{equation}
where, for each $i\geq 1$, $n_i(\m)$ denotes the number of vertices of degree $2i$ of $\m$. For $\m\in \bc_g$, we clearly have $\sum_{i}in_i(\m)=|E(\m)|$, so that by Euler's formula: 
\begin{equation}\label{eq:EulerBicol}
  |F_\bullet(\m)|+|F_\circ(\m)| = 2-2g+\sum_{i\geq 1}(i-1)n_i(\m).
\end{equation}
This implies that the generating series of these maps without vertices of degree 2, i.e. the evaluation $BC_g(z_\bullet,z_\circ;0,d_2,\ldots;t)$, is a formal power series in $z_\bullet$ and $z_\circ$ with polynomial coefficients in $t$ and in the $(d_i)_{i\geq 2}$. 
\medskip

A map is said to be $4$-valent if all its vertices have degree 4. The set of $4$-valent bicolorable maps of genus $g$ is denoted $\bc^\times_g$. Its bivariate generating series is denoted by $BC^\times(z_\bullet,z_\circ)$, and corresponds to the following specialization: 
\[
  BC_g^\times(z_\bullet,z_\circ) = BC_g(z_\bullet,z_\circ;0,1,0\ldots;1),
\]
which is indeed a formal power series in $z_\bullet$ and $z_\circ$ by the remark above.
\smallskip

\begin{figure}[t]
\centering
    \subcaptionbox{A map (plain edges) and its radial (dashed edges).\label{subfig:radial1}}
    {\includegraphics[page=1,scale=0.82]{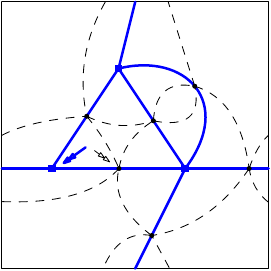}}
    \qquad\qquad
    \subcaptionbox{A map (dashed edges) and its radial (plain edges).\label{subfig:radial2}}
    {\includegraphics[page=2,scale=0.82]{images/Radial.pdf}}
    \qquad\qquad
    \subcaptionbox{The local rule and the rooting convention\label{subfig:LocalRadial}}
    {\includegraphics[page=3,scale=0.82]{images/Radial.pdf}}
\caption{\label{fig:radial}Illustration of the radial construction. The original map is represented with blue edges and square vertices and its radial by black edges and round vertices.\\
In~\subref{subfig:radial1} and~\subref{subfig:radial2}, we use the classical representation of the torus of genus 1 as a square where  both pairs of its opposite sides must be identified. In~\subref{subfig:radial2}, the faces of the radial map are colored in black (actually gray for sake of visibility) and white, following our coloring convention.}    
\end{figure}

This family of maps holds special interest thanks to the \emph{radial construction} introduced by Tutte in~\cite{Tut63}. This construction maps bijectively $\mathcal{M}_g$ to $\bc^\times_g$ in the following way (illustrated in Figure~\ref{fig:radial}). Fix $\m\in \mathcal{M}_g$. The radial $\r$ of $\m$ is then defined as follows. 
First, the vertices of $\r$ correspond to the edges of $\m$. Then, each corner of $\m$ gives rise to an edge in $\r$. More precisely, for $(h,h')\in C(\m)$, let $e,e'\in E(\m)$ be such that $h\subset e$ and $h'\subset e'$. Write $v_e$ and $v_{e'}$ for the corresponding vertices in $\r$ and add an edge in $\r$ between $v_e$ and $v_{e'}$. Finally, we root $\r$ at the corner adjacent to the face $f$ corresponding to the root vertex of $\m$, and just after the edge of $\r$ corresponding to $\rho_\m$, in clockwise order around $f$.

It can be proved that $\r$ is a map of genus $g$. Moreover, since every edge of $\m$ is incident to 4 corners, $\r$ is 4-valent. Faces of $\r$ correspond either to a vertex or a face of $\m$. For $c\in C(\m)$, the corresponding edge of $\r$ separates the faces of $\r$ corresponding to the vertex and to the face adjacent to $c$. By consequence, $\r$ is bicolorable and the rooting convention ensures that vertices and faces of $\m$ correspond respectively to black and white faces of $\r$. Hence, the radial construction has the following enumerative corollary, which is essential to our result:

\begin{proposition}\label{prop:radial}
The radial construction is a bijection between $\mathcal{M}_g$ and $\bc^\times_g$. Moreover, fix $\m\in \mathcal{M}_g$ and denote by $\r$ its image. Then, the radial construction induces a bijection between $V(\m)$ and $F_\bullet(\r)$ on one hand and between $F(\m)$ and $F_\circ(\r)$ on the other hand. 

Therefore, it yields the following equality of generating series: 
\begin{equation}
M_g(z_\bullet,z_\circ)=BC^\times_g(z_\bullet,z_\circ)
\end{equation}
\end{proposition}

\subsection{Unicellular maps, cores, schemes}\label{sub:unicellular}
A \emph{unicellular map} is a map with only one face. In particular, a unicellular map of genus 0 is a plane tree. 

The \emph{clockwise contour} of a unicellular map is the sequence of corners encountered while following the boundary of its unique face starting from its root corner, while keeping the interior of the face on its right. It corresponds to following the boundary of the disk homeomorphic to the face in clockwise direction, and note that it means turning in \emph{counterclockwise} direction around the vertices.
\medskip

The \emph{core} of a unicellular map $\u$ -- denoted $\Psi(\u)$
 -- is the map obtained from $\u$ by removing iteratively its vertices of degree one together with their incident edges. It is easy to see that the collection of deleted edges forms a forest of plane trees. Hence, a unicellular map can be canonically decomposed into its core and this forest of plane trees, see Figure~\ref{fig:schemes}\subref{subfig:coreDec}. The operation of that associates its core to a unicellular map $\u$ is called the \emph{pruning} of $\u$.

Next, the \emph{scheme} of $\u$ -- denoted $\Theta(\u)$ -- is the map obtained from $\Psi(\u)$ by additionally removing its vertices of degree 2. More precisely, for any $v$ a vertex of degree $2$ in $\Psi(\u)$, denote by $x,y$ its two neighbours. Then, merge the two former edges $\{v,x\}$ and $\{v,y\}$ into an edge $\{x,y\}$ and delete $v$, see Figure~\ref{fig:schemes}\subref{subfig:schemeHexa}.  
. 
\smallskip

By extension, a \emph{core} is a unicellular map in which all the vertices have degree not smaller than 2 and a \emph{scheme} is a unicellular map in which all the vertices have degree not smaller than 3.

\begin{figure}[t]
  \centering \subcaptionbox{Canonical decomposition of a unicellular map.\label{subfig:coreDec}}{\;\includegraphics[page=8,width=0.28\linewidth]{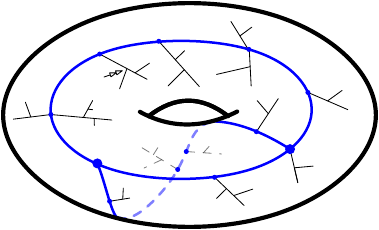}}\qquad
  \subcaptionbox{The scheme of genus 1 with 2 vertices.\label{subfig:schemeHexa}}{\;\includegraphics[page=3,width=0.28\linewidth]{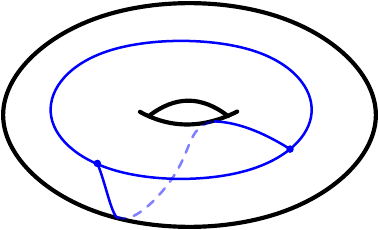}}\qquad
  \subcaptionbox{The scheme of genus 1 with 1 vertex.\label{subfig:schemeCarre}}{\;\includegraphics[page=4,width=0.28\linewidth]{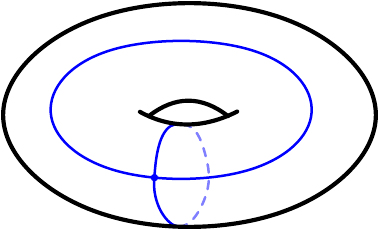}}
 \caption{In~\subref{subfig:coreDec}, we represent schematically the decomposition of a unicellular map into a forest of trees (represented in thin black edges) and a core (represented in fat blue edges). Its scheme is represented in~\subref{subfig:schemeHexa}. In~\subref{subfig:schemeCarre}, we represent the other possible scheme of genus 1.}
  \label{fig:schemes}
\end{figure}

\begin{remark} A simple computation shows that in genus 1, a scheme has either 1 or 2 vertices. Both possibilities are illustrated in Figure~\ref{fig:schemes}\subref{subfig:schemeHexa} and \ref{fig:schemes}\subref{subfig:schemeCarre}. When the scheme of a unicellular map $\u$ of genus 1 has 2 vertices, it is convenient to represent $\u$ inside an hexagon whose edges need then to be identified. In that case, the vertices of the hexagon correspond to the vertices of $\Theta(\u)$. By extension, maps of genus 1 are often represented inside an hexagon: for instance the map represented in Figure~\ref{fig:orientations} is the same as the one represented in Figure~\ref{fig:Tore}\subref{subfig:MapsTore}.
\end{remark}

\subsection{Orientations}\label{sec:orientations}
An \emph{orientation} of a map is an orientation of all its edges. An \emph{oriented map} is a map endowed with an orientation. 
For the remainder of the section, let $\m$ be a fixed oriented map.
A face of $\m$ is said to be \emph{clockwise} (resp. \emph{counterclockwise}) if its boundary forms a directed cycle and such that $\rho_\m$ lies on its left (resp. on its right). 

\begin{definition}
The orientation of $\m$ is \emph{bicolorable} if along any simple oriented loop $\ell$ drawn on the underlying surface such that $\ell \cap V(\m)=\emptyset$, there are as many edges (with multiplicity) crossing $\ell$ towards its right and towards its left. 
\end{definition}
 It is not difficult to see that if a map admits a bicolorable orientation, then it is bicolorable. Indeed, the existence of a bicolorable orientation implies directly that any cycle in the dual map has even length. Hence the dual map is bipartite, which implies that $\m$ is bicolorable. In fact, reciprocally, any bicolorable map does admit one bicolorable orientation. We now describe one such orientation which will be quintessential to our work.

First, define the \emph{height function} $h:F(\m)\rightarrow \mathbb{Z}_{\geq 0}$ of $\m$ as the unique function that satisfies the following conditions: 
\begin{compactitem}
\item The height of the root face is equal to $0$,
\item The height of a non-root face is equal to the minimum of the height of its adjacent faces plus $1$.
\end{compactitem}
Note that for $f\in F(\m)$, $h(f)$ can equivalently be defined as the dual distance between $f$ and the root face. Note also that in the canonical coloring of the faces of a bicolorable map, black and white faces have respectively even and odd heights, see Figure~\ref{fig:orientations}\subref{subfig:dualgeodheight}. 
Since $\m$ is bicolorable, the heights of two adjacent faces differ exactly by $1$.
Then, the \emph{dual-geodesic orientation} of $\m$ is defined as the orientation such that the face on the left of an edge is the face with higher height. It is clearly a bicolorable orientation.
\medskip

\begin{figure}[t]
  \centering
   \subcaptionbox{The dual-geodesic orientation. Labels correspond to the height of faces.\label{subfig:dualgeodheight}}{\;\includegraphics[scale=0.9,page=6]{images/Rooting}}\qquad\qquad\qquad
  \subcaptionbox{The orientation is not bicolorable as witnessed by the red dashed cycle.\label{subfig:noBicolor}}{\;\includegraphics[scale=0.9,page=1]{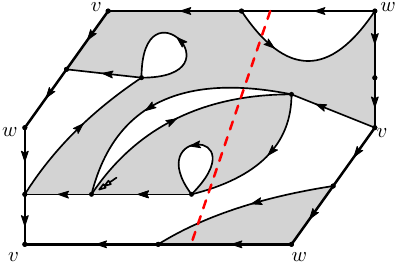}}
 \caption{The same 4-valent bicolorable map endowed with some Eulerian orientations with no clockwise face. The one in~\subref{subfig:dualgeodheight} is bicolorable and corresponds to the dual-geodesic orientation, the one in~\subref{subfig:noBicolor} is not bicolorable.}
  \label{fig:orientations}
\end{figure}

The rest of this section is not directly used in this article. However, since it is crucial to the proof of \cite[Theorem~3.14]{Lep19} on which relies heavily Theorem~\ref{thm:bijbivariate}, we include it here for sake of completeness. 

It follows from the construction that a bicolorable map endowed with its dual-geodesic orientation has no clockwise face. The following theorem stated in~\cite[Corollary 2.19]{Lep19}, and which is a consequence of the more general work of Propp \cite{Propp93}, gives another characterization of the dual-geodesic orientation, based on this observation:
\begin{theorem}\label{thm:propp}
For any bicolorable map, its dual-geodesic orientation is its unique bicolorable orientation with no clockwise face.
\end{theorem}

\begin{remark}
Consider a map $\m$, endowed with a bicolorable orientation. Fix $v\in V(\m)$ and consider a loop $\ell$ drawn on the underlying surface around $v$ small enough that no other vertex lies inside $\ell$. The definition of a bicolorable orientation implies that half the edges incident to $v$ are oriented towards it and half of them are oriented outwards it. This is precisely the definition on an Eulerian orientation. Therefore, any bicolorable orientation is in particular Eulerian. 

The dichotomy between Eulerian and bicolorable orientations is reminiscent of the dichotomy between Eulerian and bicolorable maps.
In genus 0, any Eulerian orientation is bicolorable. However, this is not true anymore in higher genus, see for instance Figure~\ref{fig:orientations}\subref{subfig:noBicolor}. The existence of the nice characterization of bicolorable orientations stated in Theorem~\ref{thm:propp} justifies the fact that bicolorable orientations are nicer to work with in higher genus.
\end{remark}

\section{Bijection for higher-genus maps, counting vertices and faces}\label{sec:bijection}

\subsection{Definition of blossoming maps}\label{sec:blossoming}
A \emph{blossoming map} is a map $\m$, with some additional unmatched half-edges, called \emph{stems}. The set of stems of $\m$ is denoted $S(\m)$
. For $s\in S(\m)$, we denote $\v(s)$ its incident vertex. Stems are oriented, and 
a stem $s$ is called a \emph{leaf} if it is oriented towards $\v(s)$ and a \emph{bud} otherwise. 
The blossoming maps we consider have almost always as many buds as leaves.

The interior map $\mathring{\m}$ of $\m$ is the map induced by $E(\m)$. For $v$ a vertex of $\m$, the \emph{interior degree} of $v$ -- denoted $\mathring{\deg}(v)$ -- is the degree of $v$ in $\mathring{\m}$. 

\begin{convention}
A rooted blossoming map is always rooted in a corner that precedes (in counterclockwise order) a bud. Equivalently, a rooted blossoming map is a blossoming map with a marked bud, called the \emph{root bud}.

In the figures, the root bud is represented by a double red arrow, see for instance Figure~\ref{fig:closurefirst}\subref{subfig:good}.
\end{convention}
In genus 0, this convention implies that rooted unicellular blossoming maps are blossoming plane trees with a marked bud. It corresponds to \emph{planted plane trees} as considered originally in \cite{Scha97}.

\subsection{Closure of blossoming maps}\label{sec:closure}
The closure of a blossoming map consists in matching its buds and leaves in a canonical way. Informally, buds and leaves are respectively seen as opening and closing parentheses and are matched as in a parenthesis word. We give in this section two (more formal) equivalent definitions of the closure operation. 

\subsubsection{First definition of closure}
Let $\u$ be a rooted blossoming unicellular map, with the same number of leaves and buds. The \emph{closure} of $\u$ is defined as follows (see Figure~\ref{fig:closurefirst}). The clockwise contour of $\u$ induces a cyclic ordering, denoted $\pct$, on the stems. Let $\beta$ be a bud (incident to $v\in V(\u)$) such that the next stem for $\pct$ after $\beta$ is a leaf $\ell$ (incident to $w\in V(\u)$). The \emph{local closure} of $\beta$ consists in merging $\beta$ and $\ell$ into an edge $\{v,w\}$ oriented from $v$ to $w$. The local closure produces a new face which lies on the left of $\{v,w\}$ and which is denoted $f(\ell)$.

\begin{figure}[t]
\centering
    \subcaptionbox{A well-rooted blossoming map,\label{subfig:good}}
    {\includegraphics[page=2,scale=0.6]{images/ClosureFirst.pdf}}
    \qquad
    \subcaptionbox{after a few local closures,\label{subfig:partial}}
    {\includegraphics[page=1,scale=0.6]{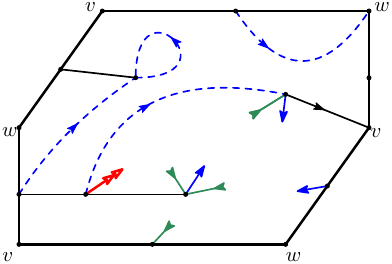}}
    \qquad
    \subcaptionbox{the full closure\label{subfig:fullclosure}}
    {\includegraphics[page=3,scale=0.6]{images/ClosureFirst.pdf}}
\caption{\label{fig:closurefirst}The closure of a well-rooted blossoming map.}    
\end{figure}

Next the \emph{closure} of $\u$ is the map obtained by performing iteratively all possible local closures (after each local closure, the ordering on the stems is obtained by removing the two stems involved in the local closure from the previous ordering). The resulting map is rooted on the corner preceding the root bud before its closure.

We can prove that the closure is well-defined in the sense that it does not depend on the order in which the local closures are performed, and that there are no remaining unmatched stems in the closure. 

\begin{definition}
A rooted blossoming unicellular map $\u$ with the same number of leaves and buds is called \emph{well-rooted} if the local closure involving its root bud can be the last local closure performed\footnote{In Schaeffer's original paper~\cite{Scha97}, well-rooted unicellular maps of genus 0 are called \emph{balanced trees}. But, in some recent works, see for instance~\cite{HDRLeveque}, the authors called \emph{balanced orientations}, some orientations which satisfy a property analogue to the one defining bicolorable orientations. To avoid confusion, we hence stick to the terminology ``well-rooted'' as introduced in~\cite{Lep19}} 

Equivalently, a rooted blossoming map is well-rooted, if there is no local closure involving a bud $\beta$ and a leaf $\ell$ such that $\beta\pct\rho_{\u}\pct\ell$. 

Equivalently, flip the orientation of the root bud of $\u$ to transform it into a ``root leaf'' and perform all the possible local closures. The resulting map will have two unmatched leaves and if one of those is the root leaf, then $\u$ is \emph{well-rooted}.
\end{definition}

Let $\u$ be a blossoming map and denote by $\chi(\u)$ its closure. There is a clear bijection between edges of $\chi(\u)$ and the reunion of $E(\u)$ and of the set of leaves of $\u$. The edges in the image of $E(\u)$ in $\chi(\u)$ are called \emph{proper edges} and the other edges of $\chi(\u)$ are called \emph{closure edges}.
In the following, we identify $E(\u)$ and the set of \emph{proper-edges} of $\chi(\u)$ and view $\mathring \u$ as a submap of~$\chi(\u)$.

\subsubsection{Closure with labels}\label{sub:closureLabels}
We now present an alternative description of the closure, slightly more formal, which is built upon a labeling of the corners. Let $\u$ be a rooted blossoming unicellular map. Recall that $C(\u)$ denotes the set of its corners.
We define the \emph{canonical labeling} $\lambda: C(\u)\rightarrow \mathbb{Z}$ of $\u$ as follows. First, $\lambda(\rho_\u)=0$. Then, we perform a clockwise contour of $\u$ and apply the following rules. Let $\kappa_1$ and $\kappa_2$ be two successive corners, then the label $\lambda(\kappa_2)$ of $\kappa_2$ is equal to: 
\begin{itemize}
	\item $\lambda(\kappa_1)$ if $\kappa_1$ and $\kappa_2$ are incident along a common edge, 
	\item $\lambda(\kappa_1)+1$ if $\kappa_1$ and $\kappa_2$ are separated by a bud,
	\item $\lambda(\kappa_1)-1$ if $\kappa_1$ and $\kappa_2$ are separated by a leaf.
\end{itemize}

The canonical labeling can be extended from corners to stems: 
for a stem $\sigma$ adjacent to $2$ corners with labels $i-1$ and $i$, we set $\lambda(\sigma)=i$, see Figure~\ref{fig:closure}\subref{subfig:goodlab}.
\medskip

We can now give an equivalent definition of the closure procedure. Consider $\u$ endowed with its canonical labeling. 
For any bud $\beta$, denote $\ell(\beta)$ the first leaf after $\beta$ (for $\preceq$) such that $\lambda(\beta)=\lambda(\ell(\beta))$. Then the \emph{closure based on labels} of $\u$ is the map obtained by merging $\beta$ and $\ell(\beta)$ into an oriented edge. 
The resulting map is rooted on the corner preceding the root bud before the closure.
\begin{figure}[t]
\centering
    \subcaptionbox{A blossoming map endowed with its canonical labeling,\label{subfig:goodlab}}    {\includegraphics[page=9,scale=.9]{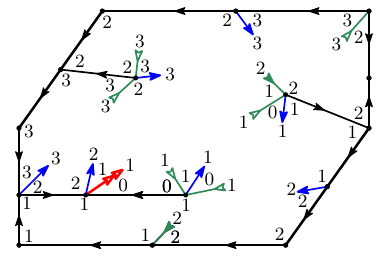}}
    \qquad\qquad
    \subcaptionbox{and after a few partial closures.\label{subfig:fullclosurelab}}
    {\includegraphics[page=17,scale=.9]{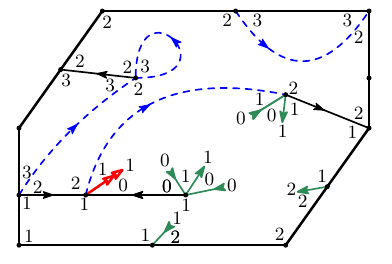}}
\caption{\label{fig:closure}Illustration of local closures based on labels. The 4-valent unicellular blossoming map $\u$ represented in \subref{subfig:goodlab} is well-rooted, well-oriented and well-labeled: it is a good map.}    
\end{figure}

The following result follows straightforwardly from the definitions:
\begin{proposition}
Let $\u$ be a rooted unicellular blossoming map, with the same number of leaves and buds. Then the closure of $\u$ and the closure based on labels of $\u$ coincide. 

Moreover, $\u$ is well-rooted if and only if all its labels are non-negative. 
\end{proposition}

\subsubsection{Closure and orientations}
For $\u$ a unicellular blossoming map with the same number of leaves and buds, observe that the closure edges of $\chi(\u)$ are naturally oriented. To endow $\chi(\u)$ with an orientation, it is therefore only necessary to prescribe the orientation of its proper edges. In other words, it is enough to orient the edges of $\mathring{\u}$.

To do so, we endow rooted unicellular maps with an orientation in the following canonical way. We perform a clockwise contour of their unique face, and orient each proper edge backwards the first time it is followed. Since the surfaces we consider are orientable, the orientation will be forward the second time the edge is followed. This yields the following definition.

\begin{definition}[well-oriented map]\label{def:wellOriented}
A blossoming unicellular map is called \emph{well-oriented} if it is endowed with the latter orientation.
\end{definition}

\subsection{Good blossoming maps and closures}
The closure operation defined in the previous section can be applied \emph{a priori} to any family of blossoming maps. Nevertheless, to obtain some interesting results, we now apply the closure to a specific family of maps. We follow the presentation of~\cite[Section~3]{Lep19}. In this section, we always assume that unicellular maps are canonically well-oriented. 

We first extend the labeling of corners to non-necessarily unicellular maps, in the following way:
\begin{definition}[well-labeled map]\label{def:wellLabeled}
A blossoming oriented map is said to be \emph{well-labeled} if its corners are labeled in such a way that: 
\begin{itemize}
\item the root corner is labeled 0, 
\item the labels of two corners lying on the same side of an edge coincide, and
\item the labels of two corners adjacent around a vertex differ by 1, in which case the higher label sits on the left of the separating edge (or stem). 
\end{itemize}
If such a labeling exists, it is called a \emph{good labeling}.
\end{definition}
Note first that an oriented map admits at most one good labeling. If $\u$ is a rooted unicellular well-oriented blossoming map, $\u$ is \emph{well-labeled} if and only if the labeling $\lambda$ of its corners defined in Section~\ref{sub:closureLabels} is a good labeling. 

\begin{claim}
Let $\m$ be an oriented map, then $\m$ admits a good labeling if and only if its underlying orientation is bicolorable.

Consequently, the closure of a rooted well-oriented and well-labeled unicellular map is a bicolorable map endowed with its dual-geodesic orientation.
\end{claim}
\begin{proof}
Let $\m$ be a (possibly blossoming) map endowed with a good labeling of its corners. Then, from the sequence of labels of corners around any fixed vertex, it is clear that the orientation is Eulerian. 

Assume now that $\m$ has no stem. It follows from the definition of a good labeling that all the corners incident to the same face has the same label, and define the label of the face to be this common value. Consider an oriented loop $\ell$ drawn on the underlying surface. Then, in the sequence of labels of faces visited by $\ell$, the label increases by one if and only if the edge crosses $\ell$ from left to right and decreases by one otherwise. It implies that the orientation of $\m$ is bicolorable. 

The closure of a well-labeled unicellular map is naturally endowed with a (good) labeling of its corners, hence the closure is bicolorable. Moreover if the unicellular map is well-oriented, then no clockwise cycle can be created during the closure so that it concludes the proof by Theorem~\ref{thm:propp}.
\end{proof}

In the rest of this section, we will state that the closure operation can be reversed and that, reciprocally, to a bicolorable map corresponds a unique well-labeled well-oriented unicellular blossoming map.
\begin{definition}\label{def:goodMaps}The set of well-rooted, well-labeled, well-oriented unicellular blossoming maps of genus $g$ is denoted by $\cO_g$. 

Moreover, the subset of $\cO_g$ restricted to 4-valent maps is denoted $\cO^\times_{g}$ and the elements of $\cO^\times_{g}$ are called \emph{good maps}.
\end{definition}
Let $\o\in \cO_g$. A leaf $\ell$ of $\o$ is said to be \emph{black} (respectively \emph{white}) if $\lambda(\ell)$ is \emph{even} (respectively \emph{odd}). The set of black and white leaves of $\o$ are respectively denoted $\ell_\bullet(\o)$ and $\ell_\circ(\o)$.
We then define the corresponding generating series as:
\begin{equation}\label{eq:defOg}{}
O_g(z_\bullet,z_\circ;d_1,d_2,\ldots;t)=\sum_{\o\in \mathcal{O}_g}t^{|E(\mathring{\o})|}z_\bullet^{|\ell_\bullet(\o)|+1}z_\circ^{|\ell_\circ(\o)|} \prod_{i\geq 1} d_i^{n_i(\o)}, 
\end{equation}
where we recall that, for each $i\geq 1$, $n_i(\m)$ denotes the number of vertices of degree $2i$ of $\o$. And similarly, we define:
\begin{equation}\label{eq:defGfGood}{}
O_g^\times(z_\bullet,z_\circ)=\sum_{\o\in \mathcal{O}^\times_g}z_\bullet^{|\ell_\bullet(\o)|+1}z_\circ^{|\ell_\circ(\o)|}.
\end{equation}

\begin{remark}\label{rem:asymetry}
The asymmetry in $z_\bullet$ and $z_\circ$ in the definition of $O_g$ comes from the asymmetry in the rooting convention of blossoming maps (and more precisely from the fact that the root bud is always matched with a white leaf). A simple way to remember this convention -- which will be useful in the sequel -- is to think of the root bud as a black leaf. So that an element of $\cO_{g}$ which contributes with a weight $z_\bullet^{n_\bullet}z_\circ^{n_\circ}$ to $O_g$ has $n_\bullet$ black leaves (including the root bud).
\end{remark}
We can now state the following theorem, which refines \cite[Theorem~3.14]{Lep19}: 
\begin{theorem}\label{thm:bijbivariate}
The closure operation is a bijection between $\cO_{g}$ and $\mathcal{BC}_{g}$, preserving the degree distribution of the vertices. Therefore, its restriction to 4-valent maps yields a bijection between $\cO^\times_{g}$ and $\mathcal{BC}^\times_{g}$. 

Moreover for $\o\in\cO_{g}$, it induces a bijection between $\ell_\circ(\o)$ and $F_\circ(\chi(\o))$ on one hand and between $\ell_\bullet(\o)$ and $F_\bullet(\chi(\o))\backslash\{f_\rho\}$ on the other hand. Hence we have the following equality of generating series:
\[
   O_g(z_\bullet,z_\circ;d_1,d_2,\ldots;t)=BC_g(z_\bullet,z_\circ;d_1,d_2,\ldots;t) \quad\text{ and }\quad O_g^\times(z_\bullet,z_\circ)=BC_g^\times(z_\bullet,z_\circ).
\]
\end{theorem}
\begin{proof}
Theorem~3.14 of \cite{Lep19} states that the closure operation is a bijection between $\cO_{g}$  and  $\mathcal{BC}_{g}$, such that the number of leaves of a map in $\cO_g$ is equal to the number of non-root faces of its closure.

Let $\o \in \cO_g$ and let  $\m=\chi(\o)$.
The definition of the closure based on labels ensures that all the corners incident to the same face $f$ of $\Phi(\o)$ have the same label, that we call \emph{height} of $f$, and denote by $h(f)$. 
Moreover, recall that for any leaf $\ell$ of $\m$, $f(\ell)$ is the face of $\Phi(\o)$ situated on the left of the closure edge associated to $\ell$. Therefore $f$ is a bijection between the set of leaves of $\m$ and $\mathcal{F}_{\Phi(\o)}\backslash f_{\rho}$, such that, additionally, the label of $\ell$ and $f(\ell)$ coincide.
Consequently, since $\m$ is bicolorable, $f$ is a bijection between black (resp. white) leaves of $o$ and black (resp. white) non-root faces of $\m$. Since the root face of $\m$ is black by definition, this concludes the proof.
\end{proof}

Combining Theorem~\ref{thm:bijbivariate} and Proposition~\ref{prop:radial} gives the following result:
\begin{corollary}\label{cor:refOp}
The closure operation combined with the radial construction gives a bijection between $\cO^\times_{g}$ and $\mathcal{M}_{g}$. And, for any $g\geq 0$, we furthermore have the following equality of generating series: 
\begin{equation}\label{eq:egMgOg}
    M_g(z_\bullet,z_\circ)=O_g^\times(z_\bullet,z_\circ).
\end{equation}
\end{corollary}

\begin{remark}
In view of the preceding corollary, to give a combinatorial proof of Theorem~\ref{thm:bivExprRat}, it is enough to prove that a similar result holds with $O_g^\times$ in place of $M_g$. That is why, from now on and until the end of the article, we only deal with 4-valent maps and in particular with good maps.
\end{remark}

\subsection{Analysis of good maps}\label{sec:schemes}
The now typical approach developed in~\cite{ChMaSc09} to enumerate unicellular maps consists in decomposing a map into a core (as defined in Section~\ref{sub:unicellular}) on which are grafted some trees. However, some technical issues need to be taken care of before applying this general strategy. The first of these issues, which already appeared in~\cite{Lep19}, comes from the fact that we consider well-rooted maps. 

To overcome this difficulty, we use some ``rerooting'' techniques as introduced for the planar case in~\cite{Scha97} and extended to higher genus in~\cite{Lep19}. A new complication appears, due to the fact that the distribution of white and black leaves is not stable via the rerooting operation. However, we prove in Theorem~\ref{thm:shortcut} that its effect averages on some subsets of maps, which is enough for our purposes.

\subsubsection{Rootable stems and unrooted maps}
We start with some definitions.
\begin{definition}[rootable stem, well-rootable stem]
A stem is called \emph{rootable} if it is either a leaf or the root bud. The set of rootable stems of a map $\m$ is denoted $\bar S(\m)$.
\end{definition}

\begin{definition}[unrooted map]\label{def:unroot}
Two rooted blossoming unicellular maps are called \emph{root-equivalent} if they have the same interior map and the same set of stems and of rootable stems (in particular, they do not necessarily have the same root). In other words, $\u$ and $\u'$ are root-equivalent if there exists $\sigma\in S^{\rho}(\u)$, such that $\u'$ can be obtained from $\u$ by turning the root bud of $\u$ into a leaf and transforming $\sigma$ into the root bud of $\u'$.

The \emph{unrooted map}
 of a rooted unicellular blossoming map $\u$ is the equivalence class of $\u$ for root-equivalence.
\end{definition}

\subsubsection{Blossoming cores and schemes}\label{subsub:blossoming}
The definitions of schemes and cores given in Section~\ref{sub:unicellular} for unicellular maps can be extended to unicellular blossoming maps. \emph{Blossoming cores} and \emph{blossoming schemes} are then defined as unicellular blossoming maps with only vertices of interior degree greater than 1 and 2 respectively.
\smallskip

\begin{figure}[t]
\centering
    \subcaptionbox{A good map,}
    {\includegraphics[page=1,scale=0.6]{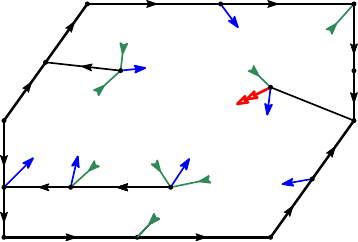}}
    \qquad\qquad
    \subcaptionbox{its (scheme-rooted) blossoming core,\label{subfig:core}}
    {\includegraphics[page=2,scale=0.6]{images/blossomingCore.pdf}}
    \qquad\qquad
    \subcaptionbox{and its blossoming scheme.\label{subfig:scheme}}
    {\includegraphics[page=3,scale=0.6]{images/blossomingCore.pdf}}
\caption{\label{fig:schemeCore}The blossoming core and scheme of a blossoming well-oriented scheme-rooted unicellular map.}    
\end{figure}
However, special care needs to be taken to handle leaves and buds when computing the core and the scheme of a unicellular blossoming map. Let $\u$ be a unicellular blossoming map. To construct its blossoming core $\Psi(\u)$, we proceed as follows. For any $v\in V(\u)$ of interior degree one, let $w$ be the unique vertex incident to $v$. Erase $v$ together with its stems, and add a stem $\sigma$ adjacent to $w$ in place of the former edge $\{v,w\}$. Repeat this procedure until there is no vertices of interior degree 1 left, see Figure~\ref{fig:schemeCore}.

In a blossoming core, vertices of interior degree 3 are called \emph{scheme vertices} and subsequently stems incident to scheme vertices are called \emph{scheme stems}. If its root bud is a scheme stem, a blossoming core is said to be \emph{scheme-rooted}.

The scheme $\Theta(\c)$ of a scheme-rooted blossoming core $\c$ is the map obtained from $\c$ by removing its vertices of interior degree 2, together with their incident stems. Vertices of interior degree $2$ in $\c$ form sequences called \emph{branches} between its scheme vertices, so that each branch of $\c$ is replaced by an edge in $\Theta(\c)$. Moreover, we identify scheme vertices and scheme stems of $\c$ with their image in $\Theta(\c)$. 
\smallskip

Note that if a unicellular map is endowed with an orientation then its core inherits its orientation. In particular, it is easy to see that the blossoming core of a well-oriented unicellular map is also well-oriented, as illustrated in Figure~\ref{fig:schemeCore}\subref{subfig:core}. Moreover, if $\c$ is a well-oriented \emph{scheme-rooted} blossoming core, the orientation of all edges inside one of its branches coincide, so that $\Theta(c)$ is also naturally endowed with an orientation. Again, it is easy to see that $\Theta(c)$ is well-oriented, see Figure~\ref{fig:schemeCore}\subref{subfig:scheme}.  
\medskip

In the rest of the paper, all the families of unicellular maps (or families in bijection with some sets of unicellular maps by \cref{thm:bijbivariate}) can be refined by specifying the unrooted scheme (or the unrooted scheme of their image by the bijection). For instance, we denote~$\mathcal{BC}^\times_{\undir{\s}}$, the subset of $\mathcal{BC}^\times$ such that $\m\in \mathcal{BC}^\times_{\undir{\s}}$ if and only if the unrooted scheme of $\chi^{-1}(\m)$ is equal to $\bar \s$.

\subsubsection{Main result}
The set of \emph{scheme-rooted well-labeled well-oriented $4$-valent blossoming cores} is denoted $\mathcal{R}$. The subset of $\mathcal{R}$ restricted to cores of genus $g$ with $n_\bullet-1$ black leaves and $n_\circ$ white leaves is denoted $\mathcal{R}_{g;n_\bullet,n_\circ}$, and the generating series $R(z_\bullet,z_\circ)$ is hence defined as: 
\[
    R(z_\bullet,z_\circ)=\sum_{n_\bullet\geq 0, n_{\circ}\geq 0}|\mathcal{R}_{g;n_\bullet,n_\circ}|z^{n_\bullet}_{\bullet}z^{n_\circ}_{\circ}.
\]
We can now state the main result of this section: 
\begin{theorem}\label{thm:shortcut}
Let $\r\in\mathcal{R}$. Then, the following equality of generating series holds: 
\begin{equation}
    O^{\times}_{\overline{\r}}(z_\bullet,z_\circ)=
    \frac{R_{\overline{\r}}(T_\bullet,T_\circ)+R_{\overline{\r}}(T_\circ,T_\bullet)}
    {2g-|\mathring n_2(\overline{\r})|},
\end{equation}
where $\mathring n_2(\overline{\r})$ denotes the set of vertices of interior degree $4$ in $\overline{\r}$ and $T_\bullet$ and $T_\circ$ are the generating series defined in~\eqref{eq:defTnoirTblanc} of Theorem~\ref{thm:bivExprRat}. 
\end{theorem}
This result is a generalization of~\cite[Lemma~4.12]{Lep19}. It requires two rerooting steps coupled with a pruning procedure, along with an analysis of the evolution of the parity of labels during the rerooting. Its proof requires to introduce a couple of additional definitions and is the purpose of the next section. 

\subsection{Proof of Theorem~\ref{thm:shortcut}}
\subsubsection{Combinatorial interpretation of $T_\bullet$ and $T_\circ$}
As mentioned above, when constructing the core of a unicellular blossoming map, the deleted edges form a forest of rooted blossoming plane trees. We now characterize these trees. 

\begin{figure}[t]
\centering
\includegraphics[scale=1]{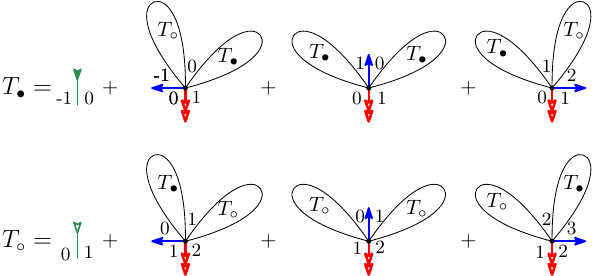}
\caption{\label{fig:decTree}Schematic representation of the decomposition of trees in $\mathcal{T}^\times$ that translate into the relations \eqref{eq:dectree} for generating series. Labels around the root vertex are indicated.}
\end{figure}

Let $\mathcal{T}^\times$
 be the set of well-oriented, well-labeled, 4-valent unicellular rooted maps of genus 0. We adopt the convention that the map reduced to a single leaf with no vertex belongs to $\mathcal{T}^\times$. 
Elements of $\cT^\times$ can alternatively be characterized as follows. They  are 4-valent blossoming plane trees with a marked bud, in which all the interior edges are oriented from a vertex to its parent. Moreover, the fact that they are well-labeled implies that each interior vertex of $\t\in \cT^\times$ is incident to exactly one bud, except for the root vertex which is incident to two buds. With this reformulation, it is clear that $\cT^\times$ corresponds exactly to the family of trees originally introduced by Schaeffer in~\cite{Scha97}.
\smallskip

Next, recall that in a well-labeled blossoming map, leaves are incident to two corners which are labeled $i$ and $i-1$ in counterclockwise order around its incident vertex. If $i$ is even (respectively odd), the corresponding leaf is called a \emph{black leaf} (respectively a \emph{white leaf}). In the sequel, it will be useful to introduce the two following bivariate generating series $T_\bullet$ and $T_\circ$ to enumerate elements of $\cT^\times$: 

\begin{equation}\label{eq:dectree}
T_{\bullet}(z_\bullet,z_\circ) = z_\bullet+\sum_{\substack{t \in \mathcal{T}^\times\\t\text{ not empty}}}z_\bullet^{|\ell_\bullet(\t)|}z_\circ^{|\ell_\circ(\t)|} \qquad\text{and}\qquad
T_{\circ}(z_\bullet,z_\circ) = z_\circ+\sum_{\substack{t \in \mathcal{T}^\times\\t\text{ not empty}}}z_\bullet^{|\ell_\circ(\t)|}z_\circ^{|\ell_\bullet(\t)|}.
\end{equation}
Alternatively $T_\circ$ can be defined as the generating series of elements of $\cT^\times$, in which all the labels are shifted by 1, so that white leaves become black leaves and vice-versa. With this remark in mind, standard decomposition of trees as illustrated in~Figure~\ref{fig:decTree} allows to retrieve the following equations already given in \eqref{eq:defTnoirTblanc}:
\begin{equation}
\begin{cases}
T_\bullet &= z_\bullet +T_\bullet^2 + 2T_\circ T_\bullet\\ T_\circ&=z_{\circ}+T_{\circ}^2+2T_\circ T_{\bullet}.
\end{cases}
\end{equation}
These equations uniquely characterize $T_\bullet$ and $T_\circ$ as formal power series in two variables without constant term. Hence, this provides a combinatorial interpretation of the series involved in the rational parametrization of Theorem~\ref{thm:bivExprRat}.

\subsubsection{Scheme trunks and rerooting}
\begin{definition}
Let $\o$ be a good blossoming unicellular map. A \emph{scheme trunk} of $\o$ is either: 
\begin{itemize}
    \item an edge incident to a scheme vertex, which does not belong to the core (i.e, it is the ``trunk'' of a tree attached to the scheme),
    \item or, a rootable stem incident to a scheme vertex. 
\end{itemize}
\end{definition}

\begin{definition}
Let $\r\in \mathcal{R}$. 
A \emph{4-valent scheme-rooted tree-decorated core} is defined as $\r^\star=\big(\r,(\t_\sigma)_{\sigma\in \bar S(\r)}\big)$, where $(\t_\sigma)$ is a collection of trees in $\mathcal{T}^\times$ indexed by the rootable stems of $\r$.

We denote by $\mathcal{D}$ the set of 4-valent \emph{scheme-rooted tree-decorated cores}. 
\end{definition}

\begin{figure}
\centering
\includegraphics[width=0.9\linewidth,page=5]{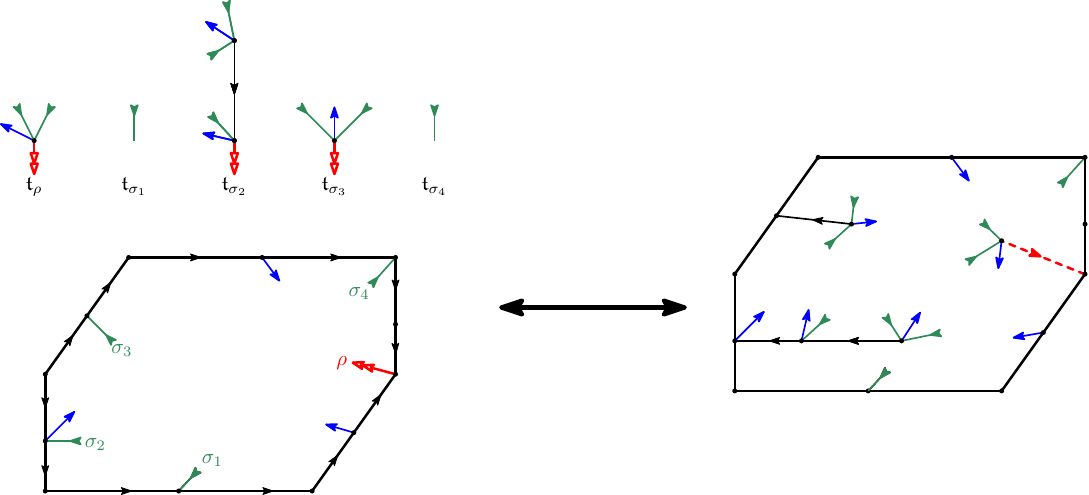}
\caption{\label{fig:decoratedCore}The correspondence between a scheme-rooted decorated core (left) and a non-rooted blossoming map with a marked scheme trunk (right)}
\end{figure}

Let $\s^\star=(\s,(\t_\ell)) \in \mathcal{D}$. A unicellular non-rooted blossoming map with a marked scheme trunk is canonically associated to $\s^\star$ by the following construction, see Figure~\ref{fig:decoratedCore}. 
For any leaf $\sigma$ of $\s$, if $\t_\sigma$ is empty, do nothing; otherwise graft $\t_{\sigma}$ to $\s$ by merging the root bud of $\t_{\sigma}$ with $\sigma$. 
Let $\rho$ be the root bud of $\s$, graft $\t_\rho$ to $\s$ by either merging its root bud with $\rho$ and marking the edge thus created if $\t_\rho$ is not empty or by only marking $\rho$ otherwise. 
In the following, we \emph{identify scheme-rooted tree-decorated cores with the corresponding blossoming maps with a marked scheme trunk}.
\smallskip

Fix $\o\in \mathcal{O}^\times$ and $\tau$ a scheme-trunk of $\o$. Then, define $\o^\tau$ to be the scheme-rooted tree-decorated core obtained by replacing the root bud of $\o$ by a leaf and by marking $\tau$, see Figures~\ref{fig:trunk}\subref{subfig:markedTrunk1} and~\ref{fig:trunk}\subref{subfig:otau}.

\begin{proposition}\label{prop:2to1}
Let $\o \in \mathcal{O}^\times$ and let $\tau$ be a scheme trunk of $\o$. Then the application which maps $(\o,\tau)$ to $\o^\tau$ is a 2-to-1 application between good maps with a marked scheme trunk and $\mathcal{D}$. 

Moreover, if we call $\tilde \o$, the other preimage of $\o^\tau$, then: 
\[
|\ell^\circ(\o)|=|\ell^\bullet(\tilde \o)|+1,\quad \text{ and }\quad  |\ell^\circ(\tilde \o)|=|\ell^\bullet( \o)|+1. 
\]
\end{proposition}

\begin{proof}
We first prove that $\o^\tau$ belongs to $\mathcal{D}$, which amounts to showing that $\Psi(\o^\tau)$ -- the scheme of $\o^\tau$ --   belongs to $\mathcal{R}$. First, it is easy to see that the blossoming core of a good map is itself well-labeled and well-oriented (when considering the restriction of the orientation of the original map to the edges that belong to the core). Therefore $\c:=\Psi(\o)$ is a well-labeled well-oriented 4-valent unicellular blossoming map (but it is not scheme-rooted).  

By a slight abuse of notation, we denote by $\tau$ the rootable stem of $\c$ which corresponds to $\tau$, see Figure~\ref{fig:trunk}\subref{subfig:Psio}. 
Then, define $\c^\tau$ to be the map obtained from $\c$ by changing the root bud into a leaf and changing $\tau$ into the root bud, in other words we ``reroot'' $\c$ at $\tau$, see Figure~\ref{fig:trunk}\subref{subfig:Psiotau}. The proof of \cite[Lemma~4.2]{Lep19} can be applied verbatim to prove that $\c^\tau$ is still well-labeled and well-oriented, hence it belongs to $\mathcal{R}$.
\smallskip

Now, let $\ell$ be the leaf which is matched to the root bud of $\o$ in its closure. Reroot $\o$ at $\ell$ (i.e. transform its root bud into a leaf and turn $\ell$ into the new root bud), and call $\tilde \o$ the resulting map. Again \cite[Lemma~4.2]{Lep19} directly implies that $\tilde \o$ is still a good map, and clearly $(\tilde \o)^\tau=\o^\tau$. Note that there might exist some non-trivial automorphisms and that $\o$ and $\tilde \o$ could be equal, but as maps with a marked scheme trunk $(\o,\tau)$ and $(\tilde \o, \tau)$ are necessarily different. So that $(\tilde \o,\tau)$ is another preimage of $\o^\tau$. 

To prove that there exists no other preimage of $\o^\tau$ in $\mathcal{O}^\times$, it suffices to notice that any other preimage is necessarily root-equivalent to $\o$ and that at most two different well-rooted maps can admit the same unrooted map. 
\smallskip

To finish the proof, observe that, generally, if we pick a rootable stem of $\o$ and reroot $\o$ at this stem, then in the new map, either all labels switch parity or none do. Moreover, they switch parity if and only if the picked rootable stem is preceded by a corner with an odd label in $\o$. It implies that all labels in $\o^\tau$ have different parity than in $\o$, see also Figure~\ref{fig:trunk}\subref{subfig:markedTrunk2}. By seeing momentarily the root bud as a black leaf (as was done in Remark~\ref{rem:asymetry}), it yields a bijection between black leaves of $\o$ and white leaves of $\tilde \o$ and vice-versa, which concludes the proof. 
\end{proof} 

\begin{figure}[t]
  \centering
    \subcaptionbox{\label{subfig:markedTrunk1}A good map $\o$ with a marked scheme trunk.}[.46\linewidth]
    {
    \includegraphics[page=1,scale=0.8]{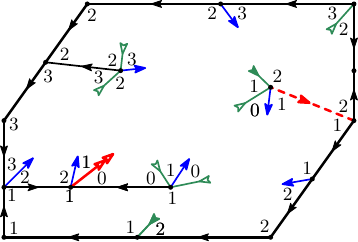}
    }\qquad
 \subcaptionbox{\label{subfig:otau}The scheme-rooted decorated core $\o^\tau$.}[.46\linewidth]
    {
    \includegraphics[page=2,scale=0.8]{images/PsiO.pdf}
    }\qquad 
    \subcaptionbox{\label{subfig:Psio}The map $\Psi(\o)$.}[.46\linewidth]
    {
    \includegraphics[page=3,scale=0.8]{images/PsiO.pdf}
    }\qquad
    \subcaptionbox{\label{subfig:Psiotau}The map $\Psi(\o)^\tau$.}[.46\linewidth]
    {
    \includegraphics[page=4,scale=0.8]{images/Rooting.pdf}
    }\qquad
   
    \subcaptionbox{\label{subfig:markedTrunk2}$\tilde \o$, the other preimage of $\o^\tau$.}[.46\linewidth]
    {
    \includegraphics[page=3,scale=0.8]{images/Rooting.pdf}
    }
\caption{\label{fig:trunk}The different maps involved in the proof of Theorem~\ref{thm:shortcut}}  
\end{figure}
\begin{proof}[Proof of Theorem~\ref{thm:shortcut}]
Let $\o\in \mathcal{O}^\times_{\bar \r}$. It follows directly from the definition of pruning, that the number of scheme trunks of $\o$ is equal to the number of rootable scheme stems of its core. By \cite[Lemma~4.7]{Lep19}, a blossoming core $\c$ has $2g-|\mathring n_2(\c)|$ rootable scheme stems (where $g$ denotes the genus of the underlying surface). Hence, $\o$ has $2g-|\mathring n_2(\Psi(\o))|=2g-|\mathring n_2(\bar \r)|$ scheme trunks and $\big(2g-|\mathring n_2(\bar \r)|\big)O^\times_{\bar \r}(z_\bullet,z_\circ)$ is the generating series of elements of $\mathcal{O}^\times_{\bar \r}$ with a marked scheme trunk. 
\smallskip

Then the result relies mostly on Proposition~\ref{prop:2to1} and on the classical relation between decomposition of combinatorial classes and composition of generating functions. However, some subtleties appear. 

First, observe that, in the pruning operation, if a dangling tree not containing the root bud is replaced by a leaf, then the labels of the corners incident to that leaf are equal to the corresponding corners before deleting the tree. It implies that the labeling of the corners of $\Psi(\o)$ have all either the same parity as the corresponding corners in $\o$ or the opposite parity. 

Fix $\tau$ a scheme-trunk of $\o$, and consider $\r^\star = (\r,(\t_\sigma)_{\sigma\in \bar S(\r)}\big)$ the scheme-decorated core associated to $\o$ by Proposition~\ref{prop:2to1}. Then $\r$ and $\Psi(\o)$ are root-equivalent and again either their corners have all the same parity or have all opposite parity. Therefore, when substituting the rootable stems of $\r$ by the trees of the family $(\t_\sigma)$, we either have: 
\begin{equation}\label{eq:1stposs}
    |\ell^\circ(\o)|=|\ell^\circ(\t_{\rho_\r})|+\sum_{\sigma\in \ell^\bullet(\r)}|\ell^\circ(\t_\sigma)|+\sum_{\sigma\in \ell^\circ(\r)}|\ell^\bullet(\t_\sigma)|
\end{equation}
or
\begin{equation}\label{eq:2ndposs}
 |\ell^\circ(\o)|=|\ell^\bullet(\t_{\rho_\r})|+\sum_{\sigma\in \ell^\bullet(\r)}|\ell^\bullet(\t_\sigma)|+\sum_{\sigma\in \ell^\circ(\r)}|\ell^\circ(\t_\sigma)|.
\end{equation}
Similar equalities hold for the number of black leaves of $\o$, but another subtlety appears. Indeed, when replacing leaves of $\r$ by trees, the root bud of $\o$ corresponds in fact to a leaf of one of these trees. More precisely, suppose that \eqref{eq:1stposs} holds, then we have (and a similar equality if \eqref{eq:2ndposs} holds): 
\[
    |\ell^\bullet(\o)|+1=|\ell^\bullet(\t_{\rho_\r})|+\sum_{\sigma\in \ell^\bullet(\r)}|\ell^\bullet(\t_\sigma)|+\sum_{\sigma\in \ell^\circ(\r)}|\ell^\circ(\t_\sigma)|.
\]
It concludes the proof, in view of the definition of the generating series $O^\times$ and $R$ and of the last assertion of Proposition~\ref{prop:2to1}.
\end{proof}

\section{Criterion for rationality via Motzkin walks}
\label{sec:motzkin}
In this section, we extend the criterion for rationality obtained in~\cite{ChMaSc09} to the bivariate case. We first introduce in Section~\ref{sec:Motz} a family of weighted paths that will appear naturally in the enumeration of blossoming cores.  In Section~\ref{sec:symRat}, we give a criterion for rationality based on these paths. 

\subsection{Weighted Motzkin paths}\label{sec:Motz}
A \emph{Motzkin walk $\w$ of length $\ell$ starting at height $h$ } is a walk in $\mathbb{Z}_{\geq 0} \times \mathbb{Z}$ starting at $(0,h)$ made of $\ell$ steps $w_1,\ldots,w_\ell$, such that $w_i\in\{(1,-1),(1,0),(1,1)\}$, for $1\leq i \leq \ell$. The set of Motzkin walks is denoted by $\mathcal{W}$. A step $w_i$ is called \emph{horizontal}, \emph{up} or \emph{down} if its second coordinate is respectively equal to $0$, $+1$ or $-1$. For $1\leq k \leq \ell$, the \emph{height} at time $k$ is equal to the ordinate of the walk when its abscissa is equal to $k$ and the \emph{height of the $k$-th step} is the height at time $k-1$. 
The \emph{increment} of $\w$ is the difference between the height at time~$\ell$ and the height at time 0. 

 A \emph{Motzkin bridge} is a Motzkin walk starting at height $0$ whose increment is equal to 0. 
 More generally, for $i,j\in \mathbb{Z}$, we denote $\mathcal{W}^{i\rightarrow j}$ the set of Motzkin walks starting at height $i$, with increment $j-i$. 
 A \emph{primitive Motzkin walk} is a Motzkin walk starting at height $0$ with increment $-1$ and such that the height of each step is non-negative. 

For $\w\in \mathcal{W}$, we denote respectively by $h(\w)$, $o(\w)$ and $e(\w)$, its number of horizontal steps, its number of non-horizontal steps with odd height (called odd steps) and its number of non-horizontal steps with even height (called even steps). For any $i,j\in \mathbb{Z}$, we introduce the following (weighted) generating series which will appear naturally in the next section: 
\begin{equation}\label{eq:defWMotzkin}
W^{i\rightarrow j}(t_\bullet,t_\circ) = \sum_{\w \in\mathcal{W}^{i\rightarrow j} }(2(t_{\bullet}+t_\circ))^{h(\w)}t_{\bullet}^{e(\w)}t_{\circ}^{o(\w)}
\end{equation}
We also consider the following generating series enumerating primitive Motzkin walks and Motzkin bridges:
\begin{align}
D_{\bullet}(t_\bullet,t_\circ) &= \sum_{\substack{\w \text{ primitive}\\ \text{Motzkin walks}}}(2(t_{\bullet}+t_\circ))^{h(\w)}t_{\bullet}^{e(\w)}t_{\circ}^{o(\w)}\\
D_{\circ}(t_\bullet,t_\circ) &= \sum_{\substack{\w \text{ primitive}\\ \text{Motzkin walks}}}(2(t_{\bullet}+t_\circ))^{h(\w)}t_{\circ}^{e(\w)}t_{\bullet}^{o(\w)}\\
B(t_\bullet,t_\circ) &= \sum_{\w \text{ Motzkin bridges}}(2(t_{\bullet}+t_\circ))^{h(\w)}t_{\bullet}^{e(\w)}t_{\circ}^{o(\w)}
\end{align}
Note, that $D_\circ$ could be equivalently defined as the weighted generating series of shifted primitive Motzkin walks, i.e starting at height 1, with increment -1 and staying positive before their last step. More generally, for $k\in \mathbb{Z}$, denote by $W_{[\geq k]}^{k \rightarrow (k-1)}(t_\bullet,t_\circ)$ the (weighted) generating series of Motzkin walks, with height not smaller than $k$, except after their last step. It is immediate to see that 
\[
W_{[\geq k]}^{k \rightarrow (k-1)}(t_\bullet,t_\circ)=
\begin{cases}
D_\bullet &\text{if }k\text{ is even}\\
D_\circ &\text{otherwise.}
\end{cases}
\]
\begin{figure}[t]
    \centering
    \includegraphics[scale=1]{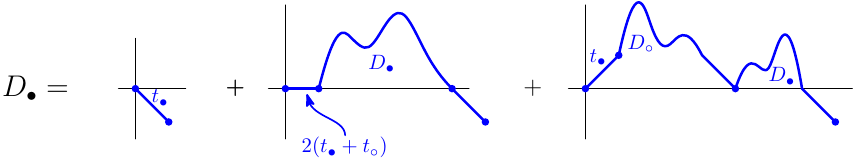}
    \caption{Schematic decomposition of primitive Motzkin walks.}
    \label{fig:decMotzkin}
\end{figure}

\begin{property}\label{prop:eqDt}
The generating series $D_\bullet$, $D_{\circ}$ and $B$ satisfy the following ``symmetric'' relations: 
\begin{equation}\label{eq:Bsym}
B(t_\bullet,t_\circ)=B(t_\circ,t_\bullet)\qquad\text{and}\qquad
t_\circ D_\bullet = t_\bullet D_\circ,
\end{equation}
and the following decomposition equations:
\begin{align}
D_\bullet&=t_\bullet+2(t_\bullet+t_\circ)D_\bullet+t_\bullet\cdot D_\circ D_\bullet,\\
D_\circ&=t_\circ+2(t_\bullet+t_\circ)D_\circ+t_\circ\cdot D_\bullet D_\circ,\\
B&=1+2(t_\bullet+t_\circ)B+(t_\bullet D_\circ + t_\circ D_\bullet)B.
\end{align}
\end{property}
\begin{proof}
In a Motzkin bridge, the number of odd up-steps is clearly equal to the number of even down-steps and the number of even up-steps is equal to the number of odd down-steps, so that the number of odd steps is equal to the number of even steps.
Therefore, a symmetry around a vertical axis gives the first equality. 
The same argument applied to positive bridges (enumerated by $\frac{D_\bullet}{t_\bullet}$) proves the second equality. 
\smallskip

It follows directly from \eqref{eq:Bsym}, that for any $k\in \mathbb{Z}$, $B(t_\bullet,t_\circ)=W^{k\rightarrow k}(t_\bullet,t_\circ)$. 
Hence, the three decomposition equations are obtained by the first-passage to 0 classical decomposition (see e.g.~\cite[Section~I.5.3]{FlaSed}), illustrated in Figure~\ref{fig:decMotzkin} in the case of $D_\bullet$.
\end{proof}

\subsection{Symmetry and rationality}\label{sec:symRat}
The relations stated in~\cref{prop:eqDt} imply that $t_\bullet$ and $t_\circ$ might be viewed as rational functions in $D_\bullet$ and $D_\circ$ defined as follows:
\begin{align}
\label{eq:tbul}
t_\bullet (D_\bullet,D_\circ)&=\frac{1}{D_\circ+2\left( \frac{D_\circ}{D_\bullet}+1\right) +\frac{1}{D_\bullet}},\\
\label{eq:tcir}
t_\circ (D_\bullet,D_\circ)&=\frac{1}{D_\bullet+2\left( \frac{D_\bullet}{D_\circ}+1\right) +\frac{1}{D_\circ}}.
\end{align} 
This expression implies in particular that any function in the two variables $t_\bullet$ and $t_{\circ}$ can be seen as a function in $D_{\bullet}$ and $D_\bullet$.

\begin{definition}
Let $f(x,y)$ be a function of two variables, we say that
\begin{itemize}
\item  $f$ is \emph{symmetric} if $f(x,y)=f(y,x)$,
\item $f$ is \emph{$\parallel$-symmetric} if 
$f(x,y)=f(x^{-1} ,y^{-1})$.
\item $f$ is \emph{$\times$-symmetric} if 
$f(x,y)=f(y^{-1} ,x^{-1})$.
\end{itemize}
We denote by $f^{\yy}$ the function defined by $f^{\yy}=f(x,y)+f(y,x)$, so that $f^{\yy}$ is always symmetric and is $\parallel$-symmetric if and only if is $\times$-symmetric.

We also denote by $\overline{f}^\parallel$ the function such that $\overline{f}^\parallel(x,y)=f(x^{-1} ,y^{-1})$, so that $f$ is $\parallel$-symmetric if and only if $f=\overline{f}^\parallel$.
\end{definition}
For instance, it is clear from \eqref{eq:tbul} and~\eqref{eq:tcir} that $t_\bullet$ and $t_\circ$ are $\times$-symmetric, when considered as functions of $D_\bullet$ and $D_\circ$.
\medskip

Based on these definitions, we get the following nice criterion for rationality in $z_\circ$ and $z_\bullet$, extending~\cite[Lemma~9]{ChMaSc09}.
\begin{lemma}
\label{lem:criterion}
Let $f$ be a symmetric function and write $F$ for the function such that $F(D_\circ,D_\bullet)=f(t_\circ,t_\bullet)$. Then $F$ is also symmetric. 

Moreover, the two following properties are equivalent: 
\begin{enumerate}[(i)]
  \item $f$ is a rational function.
  \item $F$ is a rational function and is $\parallel$-symmetric
\end{enumerate}
\end{lemma}

\begin{proof}
We start by proving that $F$ is symmetric. The definition of $D_\circ$ and $D_\bullet$ immediately implies that $D_\circ(t_\circ,t_\bullet)=D_\bullet(t_\bullet,t_\circ)$. We then have: 
\[
  F(D_\circ,D_\bullet)=f(t_{\circ},t_{\bullet})= f(t_{\bullet},t_{\circ})=F(D_\bullet,D_\circ),
\]
hence $F$ is symmetric. The implication $(i)\implies (ii)$ is clear from the expression of $t_\circ$ and $t_\bullet$ in terms of $D_\circ$ and $D_\bullet$. 
\smallskip

We now turn our attention to the implication $(ii)\implies (i)$. Let $P$ and $Q$ be two bivariate polynomials such that $F=P/Q$. Since $F$ is symmetric and $\parallel$-symmetric, we can write:
\begin{align}
  F(D_\circ,D_\bullet)&=\frac{1}{4}\left(\frac{P(D_{\circ},D_\bullet)}{Q(D_{\circ},D_\bullet)} + \frac{P(1/D_{\circ},1/D_\bullet)}{Q(1/D_{\circ},1/D_\bullet)} +\frac{P(1/D_{\bullet},1/D_\circ)}{Q(1/D_{\bullet},1/D_\circ)}+\frac{P(D_{\bullet},D_\circ)}{Q(D_{\bullet},D_\circ)}\right)\\
  &=\frac{1}{4}\frac{\tilde P(D_{\circ},D_{\bullet},1/D_\circ,1/D_\bullet)}{Q(D_{\circ},D_\bullet)Q(D_{\bullet},D_\circ)Q(1/D_{\circ},1/D_\bullet)Q(1/D_{\bullet},1/D_\circ)}, 
\end{align}
where $\tilde P$ is a polynomial. Hence $F$ can be written as $\tilde P/\tilde Q$, where $\tilde P$ and $\tilde Q$ are two polynomials in $D_\circ$, $D_\bullet$, $D_\circ^{-1}$ and $D_\bullet^{-1}$, which are both $\parallel$-symmetric and $\times$-symmetric. 
The set of such polynomials 
admits $(D_{i,j})$ as a base, where we set:
\begin{equation}
D_{i,j} := D_\circ^iD_\bullet^j+ D_\bullet^iD_\circ^j+D_\circ^{-j}D_\bullet^{-i}+ D_\bullet^{-j}D_\circ^{-i}, \quad \text{for}\quad i\in \mathbb{Z}_{\geq 0} \text{ and }j\in \{-i,\ldots,i\}.
\end{equation}

To conclude the proof, it is enough to establish that $D_{i,j}$ can be written as a rational function of $t_\bullet,t_\circ$, for any $i$ and $j$ as above. To do so, we proceed by induction on $i$. 
We first check by direct computation that $D_{i,j}$ satisfies the desired property for $0\leq i\leq 2$ and $-i\leq j \leq i$. For instance, we have: 
\begin{equation}
D_{1,0}= (1-2t_\circ-2t_\bullet)\Big(\frac{1}{t_\circ}+\frac{1}{t_\bullet}\Big)\qquad \text{ and }\qquad
D_{1,1} = 2\frac{4t_\circ^2+6 t_\circ t_\bullet+4 t_\bullet^2-4 t_\circ-4 t_\bullet+1}{t_\circ t_\bullet}.
\end{equation}
Next, fix $i\geq 3$ and $0\leq j\leq i$ and observe that: 
\begin{align}
    D_{i,j} &= \Big(D_\circ D_\bullet+\frac{1}{D_\circ D_\bullet}\Big)D_{i-1,j-1}-D_{i-2,j-2},\\
    &= \frac12 D_{1,1}D_{i-1,j-1}-D_{i-2,j-2}.
\end{align}
If $j\geq 2$, we can directly apply the induction hypothesis to conclude that $D_{i,j}$ can be written as a rational function of $t_\circ$ and $t_\bullet$. For $0\leq j<2$, either $|i-2|\geq|j-2|$ and we can proceed by induction, or $|i-2|<|j-2|$, meaning that $i=3$ and $j=0$ so that $D_{i-2,j-2}=D_{1,-2}=D_{2,1}$, which corresponds to one of the base cases. 

Now for negative values of $j$, i.e. when $-i\leq j<0$, we can similarly write :
\[
    D_{i,j} = \Big(\frac{D_\circ}{D_\bullet}+\frac{D_\bullet}{D_\circ}\Big)D_{i-1,j+1}-D_{i-2,j+2},
\]
and conclude in the same way. 
\end{proof}

\section{Rationality of the generating series of cores and schemes}\label{sec:genSeries}
The main point of this section is to introduce enough material to state Theorem~\ref{thm:MirrorR}, which is the main result of this article and establishes a refined version of the rationality scheme given in Theorem~\ref{thm:bivExprRat}. To state this result, we first group together blossoming cores that share the same ``labeled scheme''; it allows us to give an identity between their generating series and the generating series of typed Motzkin walks in Lemma~\ref{lem:decompCore}. Then, we go one step further and group together blossoming cores that admit the same ``unlabeled scheme'', which will lead to the refined rationality scheme. 

\subsection{Decomposition of cores according to their labeled scheme.}\label{sec:brancheMotz}
In this section, we prove some identities between the generating series of blossoming cores with a fixed ``labeled scheme'' and the generating series of typed Motzkin walks.

\begin{definition}
A rooted well-oriented blossoming scheme, equipped with a labeling of its corners, is called a \emph{labeled scheme} if:
\begin{itemize}
 \item its root corner is labeled 0,
 \item and the labels of adjacent corners incident to the same vertex differ by $1$, with the corner of higher label on the left of the separating half-edge.
 \end{itemize} 
The set of labeled schemes is denoted by $\mathcal{L}$.
\end{definition}
Note that a labeled scheme is not necessarily well-labeled because labels of corners adjacent along an edge may differ, see for instance Figure~\ref{fig:decorated}.

For $\r \in \mathcal{R}$, recall from Section~\ref{subsub:blossoming} that $\Theta(\r)$ denotes the scheme of $\r$. Additionally, if $\r$ is endowed with a labeling of its corners, we let corners of $\Theta(\r)$ inherit the labeling of their corresponding corners in $\r$. So that, by definition if $\r$ is well-labeled, then $\Theta(\r)\in \mathcal{L}$. For $\l \in \cL$, we set by 
\begin{equation}\label{eq:defRl}
\mathcal{R}_\l:=\{\r \in \mathcal{R} \quad\text{s.t.}\quad \Theta(\r)=\l\}.
\end{equation}

\begin{remark}
In Theorem~\ref{thm:shortcut}, we grouped together ``good maps'' with the same unrooted core, thus obtaining a connection between the two generating series $O^\times_{\bar \r}$ and $R_{\bar \r}$, for any $\r\in \cR$. To apply the criterion given in \cref{lem:criterion}, we go in a slightly different direction and now group together elements of $\mathcal{R}$ which admit the same labeled scheme. For any $\l \in \mathcal{L}$, it allows us to obtain in Lemma~\ref{lem:decompCore} a relatively explicit formula -- as a function of $D_\circ$, $D_\bullet$ and $B$ -- for the generating function $R_{\l}$. 
\end{remark}

Fix $\l \in \cL$. To construct an element of $\mathcal{R}_\l$ from $\l$, we need to replace each edge of $\l$ by a sequence of vertices of inner degree 2, in such a way that the labels around the scheme vertices after substitution coincide with the labels in the corners of $\l$. 
\begin{figure}[t]
    \centering
    \includegraphics[page=9]{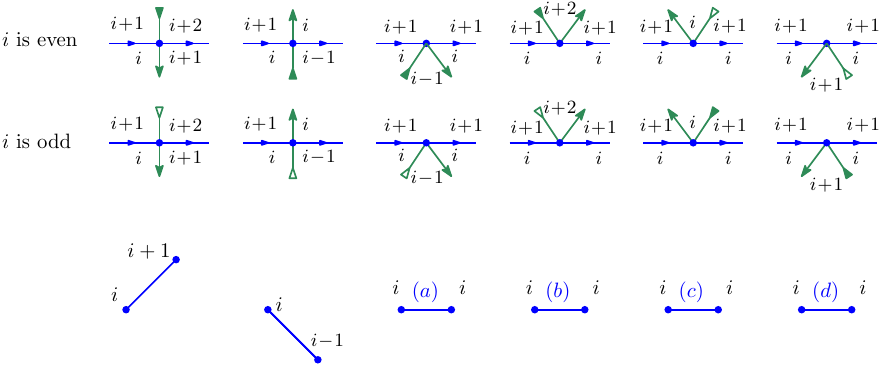}
    \caption{Illustration of the 6 types of branch vertices (top), the distinction between $i$ even and odd enables us to indicate the color of the leaves on the branches.\\ On the bottom, we represent the corresponding steps in the typed Motzkin paths where the heights of the vertices are indicated. The heights always correspond to the labels situated on the right of the oriented edges in the map.}
    \label{fig:corrVertStep}
\end{figure}

In an element of $\cR$, the vertices of interior degree $2$ can be of 6 different types, as illustrated in \cref{fig:corrVertStep} (top). When following a branch in the direction given by its orientation, the first type of vertex implies an increase of $1$ in the labels, the second one a decrease of $1$, and the last $4$ do not change the value of the labels along the branch. Each of the latter ones is given a type in $\{a,b,c,d\}$, as described in~\cref{fig:corrVertStep}. Define the set of \emph{typed Motzkin paths} to be the set of Motzkin walks such that each of their horizontal steps is decorated by a type in $\{a,b,c,d\}$. 

For $e=\{u,v\}\in E(\l)$ oriented from $u$ to $v$ (in the canonical orientation of $\l$), we denote respectively by $\lamZ(e)$ and $\lamU(e)$
the labels of the corners on the right of $u$ and respectively incident to $u$ and to $v$, see Figure~\ref{fig:decorated}. Hence, for $\mathcal{R}_\l$, a branch with $k$ edges in $\r$ corresponding to the edge $e$ of $\l$ can be encoded by a typed Motzkin walk of length $k-1$ starting at height $\lamZ(e)$ and ending at height $\lamU(e)$. Hence, we have:
\begin{claim}\label{cla:appPhi}
Define a \emph{decorated labeled scheme} $\l^\star$ as a couple $\Big(\l,(\w_e)_{e\in E(\l)}\Big)$, such that $\l\in \mathcal{L}$ and for each $e\in E(\l)$, $\w_e$ is a typed Motzkin path, which goes from height $\lamZ(e)$ to height $\lamU(e)$.

Then the application $\phi$ that maps $\l^\star$ to the blossoming core obtained by replacing each edge $e\in E(\l)$ by the branch encoded by $\w_e$ is a bijection between decorated labeled schemes with $\l$ as labeled scheme and $\mathcal{R}_\l$.
\end{claim}

For an edge $e\in E(\l)$, $e$ is said to be \emph{increasing} if $\lamZ(e)\leq \lamU(e)$ and decreasing otherwise.The sets of increasing and decreasing edges of $\l$ are respectively denoted by $\Eup(\l)$  and $\Edo(\l)$.

\begin{lemma}\label{lem:decompCore}
For any $\l\in \cL$, recall that $\bar S(\l)$ denotes the set of its rootable stems, then we have the following equality:
\begin{equation}
  R_\l(t_\bullet, t_\circ)=
  \prod_{\sigma\in \bar S(\l)}\Delta_{\lambda(\sigma)}^{\lambda(\sigma)+1}(t_\bullet,t_\circ)\cdot
  \prod_{e\in \Eup(\l)}B\cdot\Delta_{\lamZ(e)}^{\lamU(e)}(D_\bullet,D_\circ)\cdot
  \prod_{e\in \Edo(\l)}B\cdot\Delta_{\lamU(e)}^{\lamZ(e)}(D_\circ,D_\bullet),
\end{equation}
where for $i\leq j\in \mathbb{Z}_{\geq 0}$, we set:
\begin{equation}\label{eq:defDeltaPairImpair}
\Delta_i^j(x,y)\coloneqq x^{|[i,j[\cap (2\mathbb{Z})|} y^{|[i,j[\cap (2\mathbb{Z}+1)|} \quad\text{ for }i<j \qquad \text{ and }\qquad \Delta_i^i(x,y)=1.
\end{equation}

\end{lemma}

\begin{proof} The idea of the proof is completely standard and amounts to adding up the contributions of the rootable stems incident to $\l$ and of the ones coming from the substitution of edges by branches. However, some special care has to be taken in order to track the colors of stems or equivalently the parity of their labels. 

Let $\l^\star = \Big(\l,(\w_e)\Big)$ be a decorated labeled scheme. The rootable stems of $\r:=\phi(\l^\star)$ are either rootable stems of $\l$ or leaves created when replacing the edge $e$ of $\l$ by the branch encoded by $\w_e$. It is clear from Figure~\ref{fig:corrVertStep} that for any $e\in E(\l)$: 
\begin{itemize}
  \item each odd step of $\w_e$ creates one white leaf in $\r$,
  \item each even step of $\w_e$ creates one black leaf in $\r$,
  \item each horizontal step of $\w_e$ of type $(a)$ or $(b)$ creates one white leaf if $i$ is odd (resp. one black leaf if $i$ is even) in $\r$
  \item each horizontal step of $\w_e$ of type $(c)$ or $(d)$ creates one black leaf if $i$ is odd (resp. one white leaf if $i$ is even) in $\r$,
\end{itemize}
Recall the definition of the generating series $W^{i\rightarrow j}$ given in~\eqref{eq:defWMotzkin}, we can write: 
\begin{equation}
  R_\l(t_\bullet, t_\circ)=\prod_{\sigma\in \bar S(\l)}\Delta_{\lambda(\sigma)}^{\lambda(\sigma)+1}(t_\bullet,t_\circ)\cdot
   \prod_{e\in E(\l)} W^{\lamZ(e)\rightarrow \lamU(e)}(t_\bullet,t_\circ).
\end{equation}
Fix now $e\in\Edo(\l)$, then a typed Motzkin path from  $\lamZ:=\lamZ(e)$ to $\lamU:=\lamU(e)$ can be decomposed into a $(\lamZ-\lamU)$-tuple of primitive typed Motzkin paths followed by a typed Motzkin bridge. It implies that for $e\in \Edo(\l)$: 
\[
  W^{\lamZ\rightarrow \lamU}(t_\bullet,t_\circ) = B(t_\bullet,t_\circ)\prod_{k=\lamU+1}^{\lamZ} W_{[\geq k]}^{k \rightarrow (k-1)}(t_\bullet,t_\circ), 
\]
Hence, for $e\in \Edo(\l)$: 
\[
  W^{\lamZ\rightarrow \lamU}(t_\bullet,t_\circ) = B(t_\bullet,t_\circ)\cdot \Delta_{\lamU}^{\lamZ}(D_\circ,D_\bullet).
  \]
A similar decomposition when $e$ belongs to $\Eup(\l)$ concludes the proof of the lemma.
\end{proof}
\begin{figure}[t]
\centering
\includegraphics[width=1\linewidth,page=4]{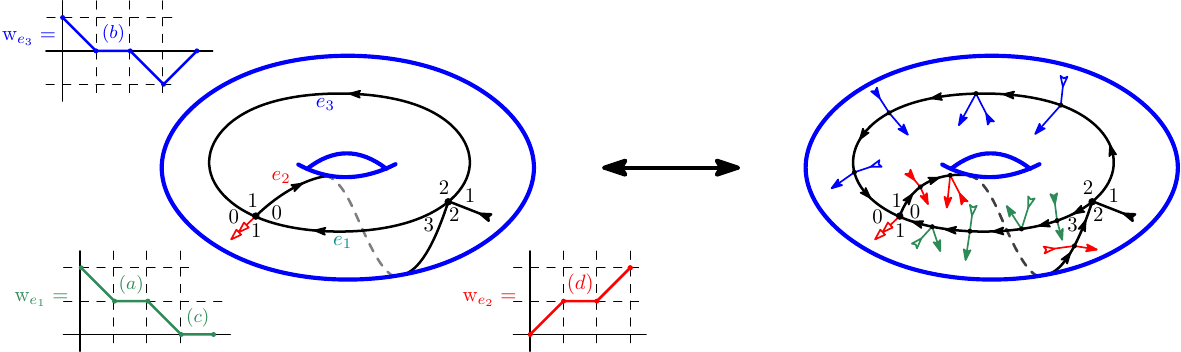}
\caption{\label{fig:decorated}The correspondence between a decorated labeled scheme $\l^\star$ (left) and a blossoming core (right).\\ In this example, $\lamZ_\l(e_1)=2$ and $\lamU_\l(e_1)=0$. Moreover, $\Eup(\l)=\{e_2\}$ and $\Edo(\l)=\{e_1,e_3\}$.}
\end{figure}
\begin{corollary}\label{cor:RratD}
For any $\l\in \cL$, $R_\l$ is a rational function of $D_\bullet$ and $D_\circ$.
\end{corollary}
\begin{proof}
Thanks to Lemma~\ref{lem:decompCore} and the expressions for $t_\bullet$ and $t_\circ$ given in~\eqref{eq:tbul} and \eqref{eq:tcir}, it is easy to see that $R_\l$ is a rational function of $B$, $D_\circ$, and $D_\bullet$. Moreover, a little of algebra based on Property~\ref{prop:eqDt} and Equations \eqref{eq:tbul} and \eqref{eq:tcir} gives the following expression for $B$:
\begin{equation}\label{eq:BD}
  B = \frac{1+2(D_\bullet+D_\circ) + D_\bullet D_\circ}{1-  D_\bullet D_\circ},
\end{equation}
so that $R_\l$ is also rational as a function of $D_\bullet$ and $D_\circ$ only.
\end{proof}

Since for any genus $g\geq 1$, the number of labeled schemes of genus $g$ is infinite, we explain in the next section how to further regrouping elements of $\mathcal{R}$, so as to be able to apply criterion of~\cref{lem:criterion}.

\subsection{Unlabeled schemes, consistent namings and height ordering}\label{sec:prel}
Let $\l\in\mathcal{L}$ be a labeled scheme, then its \emph{unlabeled scheme} -- denoted $\us(\l)$
 -- is (unsurprisingly) defined as the scheme obtained from $\l$ after erasing its labels. The set of unlabeled schemes is denoted $\mathcal{S}$.

\begin{figure}[t]
\centering
    \subcaptionbox{A well-oriented labeled scheme,\label{subfig:schemegenus2}}    {\includegraphics[page=1,scale=.77]{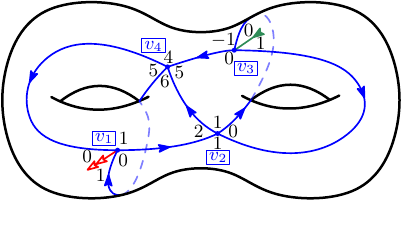}}
    \qquad
    \subcaptionbox{with the height of its vertices and its relative labels. \label{subfig:relative}}
    {\includegraphics[page=2,scale=.77]{images/Scheme_genus2_bis.pdf}}\qquad
      \subcaptionbox{The corresponding unlabeled scheme.\label{subfig:unlabeled}}
    {\includegraphics[page=3,scale=.77]{images/Scheme_genus2_bis.pdf}}
\caption{\label{fig:unlabScheme}Example of a well-oriented labeled scheme $\l$ of genus $2$ with $4$ vertices. In~\subref{subfig:schemegenus2}, $\l$ is represented classically on a surface of genus $2$ with the labeling of its corners and its orientation.
In~\subref{subfig:relative} and~\subref{subfig:unlabeled}, the same map is represented but the underlying surface is omitted.\\ In~\subref{subfig:relative} and~\subref{subfig:unlabeled}, corners are labeled by their relative labels (which could also be deduced from the canonical orientation of the edges). In~\subref{subfig:relative}, the height of each vertex is indicated by the circled red label. In~\subref{subfig:unlabeled}, the corresponding unlabeled scheme is represented.}
\end{figure}

\subsubsection{Relative labels and height}
Let $\l\in \mathcal{L}$. For $v\in V(\l)$, the \emph{height} $h(v)$
 of $v$ is defined as the minimum of the labels of its adjacent corners, i.e.:
\[
h(v) = \min\{\lambda(\kappa),\text{ for }\kappa\text{ a corner incident to }v\}.
\]
For $\kappa \in C(\l)$, the \emph{relative label} of $\kappa$ is set to be the difference $\lambda(\kappa)-h(\v(\kappa))$, where $\v(\kappa)$ denotes the vertex incident to $\kappa$. Observe that relative labels can be immediately retrieved from the orientation of the edges of $\l$ or of $\us(\l)$. 
Therefore, $\l$ can be reconstructed from $\us(\l)$ and $(h(v))_{v\in V(\l)}$, see Figure~\ref{fig:unlabScheme} and it makes sense to speak of the relative labels of an unlabeled scheme. 

\medskip 
The following fact will be useful in the sequel:
\begin{claim}\label{cla:height0}
For $\l\in \mathcal{L}$, the height of its root vertex is equal to $0$. 
\end{claim}
\begin{proof}
The only thing to prove is that the relative labels of the corners incident to the root vertex of $\l$ are all in $\{0,1\}$. Since $\l$ is well-oriented, the two edges following and preceding the root bud around its root vertex are necessarily oriented toward it, see Figure~\ref{fig:unlabScheme}\subref{subfig:schemegenus2}.
\end{proof}

\subsubsection{Offset edges and consistent naming}
In this section $\s$ is a fixed element of $\mathcal{S}$. Most definitions of this section are illustrated in Figure~\ref{fig:naming}. A half-edge of $\s$ is said to be \emph{shifted}, if its adjacent relative labels are $1$ and $2$, and \emph{unshifted} otherwise (in which case its adjacent relative labels are $0$ and $1$).

\begin{definition}[balanced, shifted, and offset edges and stems]\label{def:balancedShiftedOffset}
An edge $\{u,v\}$ of $\s$ is said to be:
\begin{compactitem}
\item \emph{unshifted} if it consists of two unshifted half-edges,
\item \emph{shifted} if it consists of two shifted half-edges,
\item \emph{offset towards $v$} if it has exactly one shifted half-edge that is incident to $v$. 
\end{compactitem}
\end{definition}

In~\cite{Lep19}, the second author obtained the following result about offset edges:
\begin{theorem}[Theorem 4.14 of \cite{Lep19}]
\label{thm:acyclicOffsetGraph}
For any $\s\in \mathcal{S}$, the directed graph of offset edges of $\s$ is acyclic.
\end{theorem}

\begin{definition}\label{def:naming}
A \emph{naming} of $\s$ is a bijection from $V(\s)$ to $[|V(\s)|]$.\footnote{A naming is simply a labeling of the vertices of $\s$ by integers from 1 to $|V(\s)|$. However, we do not use the term ``label'' to avoid any confusion with \emph{labeled schemes}, and use the term ``naming'' instead.}
A naming is called \emph{consistent} if it is a linear extension of the partial order induced by the (acyclic) oriented graph of its offsets edges. In other words, a naming $\nu$ is consistent if, for any edge $\{u,v\}$ offset towards $v$, $\nu(u)<\nu(v)$. 
\end{definition}
Theorem~\ref{thm:acyclicOffsetGraph} implies that $\s$ admits a consistent naming. 
\medskip

\begin{center}
\begin{minipage}[b]{0.9\textwidth}
\textbf{\emph{In the following, we always assume that an unlabeled scheme is endowed with a fixed consistent naming $\nu$ even if it is not explicitely mentioned.}}
\end{minipage}
\end{center}

\begin{figure}[t]
\centering
    \subcaptionbox{Labeled scheme $\l$,\label{subfig:schemeNaming}}    {\includegraphics[page=1,scale=.9]{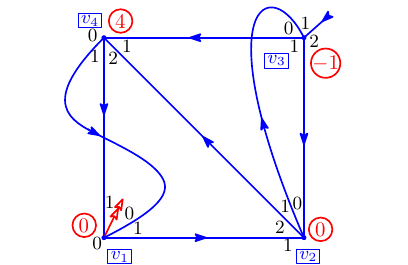}}
    \qquad\qquad\quad
    \subcaptionbox{and its unlabeled scheme $\s:=\us(\l)$.\label{subfig:schemeOffset}}
    {\includegraphics[page=2,scale=.9]{images/Offset.pdf}}
\caption{\label{fig:naming}
In~\subref{subfig:schemeNaming} this is the labeled scheme  $\l$ of Figure~\ref{fig:unlabScheme}\subref{subfig:relative}. In~\subref{subfig:schemeOffset}, this is the corresponding unlabeled scheme, where its offset edges are represented by dashed black edges with a triple arrow illustrating the direction of offset, its shifted edge is represented as a dot-dashed green edge with two triple arrows, and its shifted stem as a black half edge with a triple arrow.\\
A consistent naming $\nu$ of $\s$ is given by $\nu(v_i)=i$. Moreover, the height ordering $\pi_{\l}$ of $\l$ is equal to $\pi_{\l}(1)=v_3$, $\pi_{\l}(2)=v_1$, $\pi_{\l}(3)=v_2$ and $\pi_{\l}(4)=v_4$, since $h(v_3)< h(v_1)=h(v_2)<h(v_4)$ and $\nu(v_1)<\nu(v_2)$.
}
\end{figure}

\subsubsection{Height-order relative to a naming}
Fix $\l\in\mathcal{L}$.
, and let $\nu$ be a naming of $\us(\l)$. 
The \emph{height order} $\pi_{\l}$ of $\l$ is informally defined as the ordering of the vertices of $\l$ by their height, where two vertices of the same height are ordered thanks to the naming of $\nu$. 
More precisely, $\pi_{\l}$ is a bijection from $|V(\l)|]$ to $V(\l)$ such that for any $i<j$:
\begin{itemize} 
  \item Either $h(\pi_{\l}(i))<h(\pi_{\l}(j))$, 
  \item Or $h(\pi_{\l}(i))=h(\pi_{\l}(j))$ and $\nu(\pi_{\l}(i))<\nu(\pi_{\l}(j))$.
\end{itemize}
The inverse bijection $\pi_{\l}^{-1}:V(\l)\to [|V(\l)|]$ is called the \emph{height rank function} of $\l$. The height ordering and the height rank depend on the naming $\nu$, so that we should in fact write $\pi_{\l,\nu}$. To lighten the notation, we drop the index $\nu$ which should not cause any confusion. 
\smallskip

To enumerate labeled schemes which admit the same unlabeled scheme, we group them by height order. Write $\mathfrak{S}_\s$ for the set of bijections from $[|V(\s)|]$ to $V(\s)$. Fix $\pi \in \mathfrak{S}_\s$, then we set:
\begin{equation}
\mathcal{L}_{\s}^\pi\coloneqq\{\l\in\mathcal{L}_\s|\pi_{\l}=\pi\}
\end{equation} 
Finally, we denote by $\mathcal{R}_{\s}^\pi$ the corresponding subset of $\mathcal{R}$:
\begin{equation}\label{eq:defRSNuPi}
\mathcal{R}_{\s}^\pi\coloneqq\bigcup\limits_{\l\in\mathcal{L}_{\s}^\pi}\mathcal{R}_\l.
\end{equation}

\subsection{A key tool to prove $\parallel$-symmetry : the mirror operation}\label{sub:MainMirror}
For $\pi\in \mathfrak{S}_{\s}$, its \emph{mirror} $\overline{\pi}$ is defined by: 
 \begin{equation}\label{eq:defmirror}\overline{\pi}(k)=\pi(|V(\s)|+1-k), \quad \text{for }k\in [|V(\s)|].
 \end{equation}

The mirror operation allows us to state the following refined rationality scheme which is key in the combinatorial proof of Theorem~\ref{thm:bivExprRat}: 
\begin{theorem}\label{thm:MirrorR}
For any $\s\in \mathcal{S}$, any consistent naming $\nu$ of $\s$ and any $\pi\in \mathfrak{S}_{\s}$, we have: 
\begin{equation}
\overline{R_{\s}^\pi(D_\bullet,D_\circ)}^\parallel+\overline{R_{\s}^\pi(D_\circ,D_\bullet)}^\parallel=R_{\s}^{\overline{\pi}}(D_\bullet,D_\circ)+R_{\s}^{\overline{\pi}}(D_\circ,D_\bullet).
\end{equation}
Moreover, $R_{\s}^\pi(D_\bullet,D_\circ)$ is a rational function of $D_\bullet$ and $D_\circ$.

Hence, by Lemma~\ref{lem:criterion}, $R_{\s}(t_\bullet,t_\circ)+R_{\s}(t_\circ,t_\bullet)$ is a rational function of $t_\bullet$ and $t_\circ$.
\end{theorem}
We obtain as an immediate corollary: 
\begin{corollary}
\label{thm:symmetry}
For any $\s\in \mathcal{S}$, the series $R_{\s}^{{\yy}}(D_\bullet,D_\circ):=R_\s(D_\bullet,D_\circ)+R_\s(D_\circ,D_\bullet)$ is a rational function of $D_\bullet$ and $D_\circ$, and is $\parallel$-symmetric as a function of $D_\bullet$ and $D_\circ$.
\end{corollary}
\begin{proof}
Let $\s$ be an unlabeled scheme (endowed with a consistent naming). Then: 
\[
  R_\s=\sum_{\pi\in \mathfrak{S}_{\s}}R_{\s}^\pi.
\]
It follows directly from Theorem~\ref{thm:MirrorR} that:
\[
  \overline{R_\s^{\yy}}^\parallel=\left(\sum_{\pi\in \mathfrak{S}_{\s}}R_{\s}^{\overline{\pi}}\right)^{\yy}=R_\s^{\yy},
\]
where the last equality comes from the fact that the mirror operation is a bijection on $\mathfrak{S}_{\s}$.
\end{proof}

The proof of Theorem~\ref{thm:MirrorR} will be the purpose of the next two sections but we first explain why it concludes the bijective proof of Theorem~\ref{thm:bivExprRat}. 
\smallskip

For $\s\in \mathcal S$, its unrooted map $\mathfrak u$ (recall \cref{def:unroot}) is called an \emph{unrooted scheme}. As before, we denote by $M_{\u}(z_\bullet,z_\circ)$ the generating series of maps which correspond to the unrooted scheme $\u$. In other words, the image of these maps by the bijection given in Corollary~\ref{cor:refOp} are good maps that admit $\mathfrak{u}$ as unrooted scheme.

\begin{theorem}\label{thm:ratScheme}
For any unrooted scheme $\u$, the generating series $M_{\u}(z_\bullet,z_\circ)$ is a rational function of the series $T_\bullet$ and $T_\circ$ defined in~\eqref{eq:dectree}.

Hence, $M_g(z_\bullet,z_\circ)$ is also a rational function of $T_\bullet$ and $T_\circ$.
\end{theorem}

\begin{proof}

 \cref{lem:criterion,thm:symmetry} directly imply that the series $R_{\s}^{{\yy}}$ is a rational function of $t_\circ$ and $t_\bullet$.  Fix now $\u$ in $\mathcal{\bar S}$. The number of unlabeled schemes that admit $\u$ as unrooted scheme is equal to $2g-|\mathring n_2(\u)|$ and is hence finite. It implies that $R_{\u}^{{\yy}}$ is also rational as a series in $t_\circ$ and $t_\bullet$.

By \cref{thm:shortcut}, we deduce that $O_{\u}^{\times}$ is rational in $T_\circ$ and $T_\bullet$. And finally, by \cref{cor:refOp}, we conclude that $M_{\u}$ is a rational function of $T_\circ$ and $T_\bullet$. 

Since the number of unrooted schemes of genus $g$ is finite, it implies that $M_g(z_\bullet,z_\circ)$ is also a rational function of $T_\bullet$ and $T_\circ$, which concludes the proof.
\end{proof}

\section{The univariate case: Combinatorial proof of Theorem~\ref{thm:BC91}}
\label{sec:uniSym}
As a warm-up, we first prove a specialization of \cref{thm:MirrorR} stated in Theorem~\ref{lem:uniSymMir}, which gives a simpler combinatorial proof of Theorem~\ref{thm:BC91} than the one previously obtained in~\cite{Lep19}.

In this section, all the generating series considered are assumed to be univariate. But, with a slight abuse of notation, we keep the same notation. So that, throughout this section, we have for instance $R_\s(t)\coloneqq R_{\s}(t,t)$, $D(t):=D_\circ(t,t)=D_\bullet(t,t)$, $B(t)\coloneqq B(t,t)$ and $T(z):=T_\circ(z,z)=T_\bullet(z,z)$. Moreover, for any function $f$ of $D$, we set $\overline{f}^\parallel(D):=f(1/D)$. 

In all this section, we fix $\s\in \mathcal{S}$ and $\pi\in \mathfrak{S}_{\s}$, where we recall that $\mathfrak{S}_{\s}$ is the set of bijections from $V(\s)$ to $|V(\s)|$.

\medskip

The main result of this section is the following:
\begin{theorem}\label{lem:uniSymMir}
The generating series $R_{\s}^\pi$ is a rational function of $D$, and we have the following identity: 
\begin{equation}
\overline{R_{\s}^{\overline{\pi}}}^\parallel(D)=R_{\s}^\pi(D).
\end{equation}
\end{theorem}
As in Section~\ref{sub:MainMirror}, it gives as an immediate corollary:
\begin{corollary}
\label{thm:uniSym}{}
For any unlabeled scheme $\s\in \mathcal{S}$, the series $R_{\s}(t)$ is rational and symmetric in $D(t)$ and $1/D(t)$.
\end{corollary}

Note that a proof of \cref{thm:uniSym} is already given in \cite{Lep19}, but we give here a different proof that we were able to extend to the bivariate case.
The differences between the two proofs will be discussed in \cref{rem:diff2proofs}.
Using a univariate version of \cref{lem:criterion} due to \cite{ChMaSc09}, \cref{thm:uniSym} can be used to prove the rationality of $M_g(z)$, thus giving a combinatorial proof of~\cref{thm:BC91}.
\smallskip

The proof of Theorem~\ref{lem:uniSymMir} relies on an explicit formula for the generating series $R_{\s}^\pi$. We denote by $\tilde{R}$ the generating series of $\mathcal{R}$, where we set to 1 the weight of the scheme stems (in other words, we neglect their contribution). This is still a valid generating function since there is only a (fixed) finite number of stems on a given scheme.
We define similarly the generating series $\tilde{R}_{\s}^{\pi}$ and $\tilde{R}_\l$ for the corresponding subfamilies of $\mathcal{R}$.

Fix $\s\in\mathcal{S}$ and $\pi \in \mathfrak{S}_{\s}$, we derive in the following paragraphs a closed expression for $\tilde{R}_{\s}^{\pi}$, stated in Lemma~\ref{lem:uniEnumOrdered}. First by decomposing the family of maps $\mathcal{R}_{\s}^\pi$ according to their labeled scheme, we get:
\[ \tilde{R}_{\s}^\pi = \sum_{\l\in \mathcal{L}_{\s}^\pi}\tilde{R}_\l \]
For a given labeled scheme $\l$, the generating series $\tilde{R}_\l$ can be computed as the product of the contribution of the branches of $\l$. Indeed, setting $D_\bullet=D_\circ$ in \cref{lem:decompCore} and ignoring the contribution of the scheme rootable stems, we get:
\[\tilde{R}_\l=\prod_{e\in E(\l)} B\cdot D^{L(e)},\]
where we set $L(e):=|\lamZ_\l(e)-\lamU_\l(e)|$. For an edge $e=\{u,v\}$, the quantity $L(e)$ depends only on the heights $h_\l(u)$ and $h_\l(v)$, and whether $e$ is offset or not. Namely:
\begin{equation}
L(\{u,v\})=|h_\l(u)-h_\l(v)|+\epsilon_\l(\{u,v\}),
\end{equation} 
where: 
\begin{equation}\label{eq:defEpsilon}
\epsilon_\l(\{u,v\})=
\left\{
\begin{array}{ll}
0, &\text{if $\{u,v\}$ is not offset}\\
1, &\text{if $\{u,v\}$ is offset toward $v$ and } h_\l(v)\geq h_\l(u)\\
-1, &\text{if $\{u,v\}$ is offset toward $v$ and } h_\l(v)< h_\l(u)
\end{array}
\right. .
\end{equation}
Recall the definition of $R_{\s}^\pi$ given in~\eqref{eq:defRSNuPi}, then:
\begin{equation}\label{eq:Runi}
\tilde{R}_{\s}^\pi =B^{|E(\s)|} \cdot \sum_{\l\in \mathcal{L}_{\s}^\pi}\prod_{\substack{\{u,v\}\in E(\l)\\ \pi^{-1}(u)<\pi^{-1}(v)}} D^{h_\l(v)-h_\l(u)+\epsilon_\l(\{u,v\})},
\end{equation}
since for $\l\in\mathcal{L}_\s^\pi$, $\pi^{-1}(u)<\pi^{-1}(v)$ implies that $h_\l(u)\leq h_\l(v)$.
\medskip 

An element $\l$ of $\mathcal{L}_{\s}^\pi$ is fully characterized by a sequence $(h_\l(\pi(k))_{k\in[|V(\s)|]}$. By Claim~\ref{cla:height0}, the height of the root vertex of $\l$ must be equal to 0 so that $\l$ is equivalently characterized by the sequence $(I_k)_{k\in[|V(\s)|-1]}=(h_\l(\pi(k+1))-h_\l(\pi(k)))_{k\in[|V(\s)|-1]}$ of the differences between successive heights of the vertices $\pi(k)$ in $\l$. 

Since $\pi$ gives the ordering of vertices of $\l$ by non-decreasing height, we necessarily have that:
\[
I_k\geq 0 \quad\text{ for }k\in [|V(\s)|-1].
\]
Moreover, by definition of a consistent naming, $I_k$ can be 0 if and only if $\nu(\pi(k+1))>\nu(\pi(k))$. For $k\in [|V(\s)-1]$, we set: 
\begin{equation}\label{eq:defDelta}
\delta^k_{\pi}:=\begin{cases}
1&\text{ if }\nu(\pi(k+1))<\nu(\pi(k))\\
0&\text{ otherwise.}
\end{cases}
\end{equation}
Let $\mathcal{I}_{\s}^\pi$ be the set of sequence of integers $(I_k)_{k\in[|V(\s)|-1]}$ such that $I_k \geq \delta^k_{\pi}$ for any $k$. Then, we can rewrite \eqref{eq:Runi} as:
\begin{equation}\label{eq:Runi2}
\tilde{R}_{\s}^\pi 
=B^{|E(\s)|} \cdot \sum_{(I_k)\in \mathcal{I}_{\s}^\pi}\prod_{\stackrel{\{u,v\}\in E(\l)}{\pi^{-1}(u)<\pi^{-1}(v)}} D^{I_{\pi^{-1}(u)}+\cdots+I_{\pi^{-1}(v)-1}+\epsilon_\l(\{u,v\})}.
\end{equation}
\begin{figure}[t]
\centering
    \subcaptionbox{The unlabeled scheme $\s$ represented as a graph.\label{subfig:schemeGraph}}
    {\includegraphics[page=5,scale=1.1]{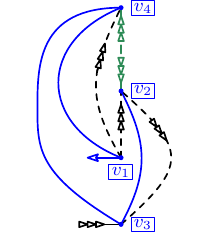}}\qquad\ 
      \subcaptionbox{Contribution to $\mathtt{C}_{\pi}^{3}$.\label{subfig:posEnclosing}}
    {\includegraphics[page=6,scale=1.1]{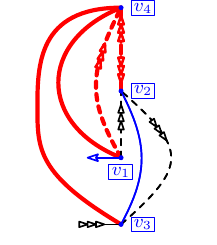}}\qquad\ 
     \subcaptionbox{Contribution to $\mathtt{O}_\pi$.\label{subfig:OverFit}}
    {\includegraphics[page=7,scale=1.1]{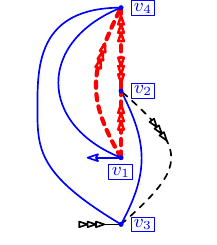}}\qquad\ 
\caption{\label{fig:unlabSchemeStatUniv}
Illustration of the quantities introduced in Definitions~\ref{def:Ck} and~\ref{def:OU} on the unlabeled scheme $\s$ of Figure~\ref{fig:naming}. We keep the same naming $\nu$ as previously, namely $\nu(v_i)=i$ for $i=1,\ldots,4$, and still represent shifted half-edges by triple arrows. We consider the height ordering $\pi$, defined by $\pi(1)=v_3$, $\pi(2)=v_1$, $\pi(3)=v_2$ and $\pi(4)=v_4$, and order vertices of $\s$ from bottom to top according to $\pi$.\\
\subref{subfig:schemeGraph}: We easily read from this representation that $\delta_\pi^1=1$ and $\delta_\pi^2=\delta_\pi^3=0$.\\
\subref{subfig:posEnclosing}: Edges contributing to $\mathtt{C}_{\pi}^{3}$ are emphasized, giving $\mathtt{C}_{\pi}^{3}=4$.\\
\subref{subfig:OverFit}: Edges contributing to $\mathtt{O}_\pi$ are emphasized.  giving $\mathtt{O}_{\pi}=3$.}
\end{figure}
To rearrange this expression, consider the representation of $\s$ in which its vertices are ordered with respect to $\pi$ from bottom to top, see Figure~\ref{fig:unlabSchemeStatUniv}. For $k\in V(\s)$, consider a horizontal line situated just above the vertex $\pi(k)$ and define $\mathtt{C}^{k}_\pi$ as the number of edges of $\s$ crossing this line, see Figure~\ref{subfig:posEnclosing}. More formally:
\begin{definition}\label{def:Ck}
For $\s$ an unlabeled scheme, $\pi \in \mathfrak{S}_\s$ and $k\in [|V(\s)|-1]$, define $\mathtt{C}^{k}_\pi$ as the number of edges $\{u,v\}\in E(\s)$ such that
\begin{itemize}
\item either $\pi^{-1}(u)\leq k<\pi^{-1}(v)$,
\item or $\pi^{-1}(v)\leq k<\pi^{-1}(u)$.
\end{itemize}
\end{definition}
Observe that the number of terms in which $I_k$ appears in \eqref{eq:Runi2} is precisely equal to $\mathtt{C}_\pi^{k}$ so that we can write~\eqref{eq:Runi2} as: 
\begin{equation}\label{eq:Runi3}
\tilde{R}_{\s}^\pi 
=B^{|E(\s)|} \cdot \sum_{(I_k)\in \mathcal{I}_{\s}^\pi}\prod_{k=1}^{|V(\s)|-1}D^{I_k\cdot \mathtt{C}^{k}_\pi}
\cdot \prod_{\{u,v\}\in E(\l)} D^{\epsilon_\l(\{u,v\})}.
\end{equation}

It only remains to deal with the contributions of $\epsilon_\l(\{u,v\})$. 
\begin{definition}\label{def:OU}
Define respectively $\mathtt{O}_{\pi}$ and $\mathtt{U}_\pi$ as the number of edges $\{u,v\}$ such that the half-edge incident to $u$ is shifted, and such that $\pi^{-1}(u)>\pi^{-1}(v)$ (respectively $\pi^{-1}(u)<\pi^{-1}(v)$), see Figure~\ref{subfig:OverFit}. \\
Letters $\mathtt{O}$ and $\mathtt{U}$ stand for ``over'' and ``under'', to indicate which extremity of the edge is shifted: the top one, or the bottom one.
\end{definition}
Since any shifted edge contributes once to $\mathtt{O}_\pi$ and once to $\mathtt{U}_\pi$, we get\footnote{Including the contribution of shifted edges to $\mathtt{O}_\pi$ and to $\mathtt{U}_\pi$ will make more sense, when we move to the bivariate enumeration where their contribution has to be taken into account.}:
\begin{equation}
\sum_{\{u,v\}\in E(\l)}\epsilon_\l(\{u,v\})=\mathtt{O}_{\pi}-\mathtt{U}_{\pi}.
\end{equation}
Therefore:
\begin{equation}
\tilde{R}_{\s}^\pi 
=B^{|E(\s)|}\cdot D^{\mathtt{O}_{\pi}-\mathtt{U}_{\pi}} \cdot \sum_{(I_k)\in \mathcal{I}_{\s}^\pi}\prod_{k=1}^{|V(\s)|-1}D^{I_k\cdot \mathtt{C}^{k}_\pi}.
\end{equation}

We have just established:
\begin{lemma}\label{lem:uniEnumOrdered}
Fix $\s\in\mathcal{S}$ and $\pi \in \mathfrak{S}_{\s}$, then:
\begin{equation}
\tilde{R}_{\s}^\pi=B^{|E(\s)|}\cdot D^{\mathtt{O}_{\pi} - \mathtt{U}_{\pi}}\cdot\prod_{k=1}^{|V(\s)|-1}\frac{D^{\delta^{k}_{\pi} \cdot \mathtt{C}^{k}_\pi}}{1-D^{\mathtt{C}^{k}_\pi}},
\end{equation}
where $\mathtt{C}^{k}_\pi$, $\mathtt{O}_\pi$, $\mathtt{U}_\pi$ and $\delta^{k}_{\pi}$ are explicit quantities depending only on $\s$, $\pi$ and $k$ and defined respectively in Definitions~\ref{def:Ck} and~\ref{def:OU} and in~\eqref{eq:defDelta}.
\end{lemma}

All the quantities introduced in the course of the previous proof interact nicely with the mirror operation defined in Section~\ref{sub:MainMirror}. It follows from the definitions that: 
\begin{fact}\label{fact:mirrorUniv}
Fix $\s\in \mathcal{S}$ and $\pi\in \mathfrak{S}_{\s}$. Recall that $\bar \pi:= k\rightarrow \pi(|V(\s)|-k+1)$, then we have:  
\begin{equation}
\mathtt{O}_{\bar \pi}=\mathtt{U}_\pi.    
\end{equation}
Additionally, for $k\in[|V(\s)|-1]$, write $\bar k = |V(s)|-k$, then:
\begin{equation}
\mathtt{C}_{\bar \pi}^{k}=\mathtt{C}_\pi^{\overline{k}},\quad \text{and}\qquad \delta_{\overline{\pi}}^k=1-\delta_\pi^{\overline{k}}.
\end{equation}
\end{fact}
We can now conclude the proof of Theorem~\ref{lem:uniSymMir}:
\begin{proof}[Proof of Theorem~\ref{lem:uniSymMir}]
Since $t$ and $B$ are rational functions of $D$, the definition of $\tilde{R}_{\s}^\pi$ and \cref{lem:uniEnumOrdered} imply that $R_{\s}^\pi$ is also rational in $D$. Moreover, by \cref{lem:uniEnumOrdered}, we have:
\begin{align}
\overline{\tilde{R}_{\s}^{\pi}}^\parallel
&=\overline{B^{|E(\s)|}\cdot D^{\mathtt{O}_{\pi}-\mathtt{U}_{\pi}}\cdot\prod_{k=1}^{|V(\s)|-1}\frac{D^{\delta^{k}_{\pi} \cdot \mathtt{C}^{k}_\pi}}{1-D^{\mathtt{C}^{k}_\pi}}}^\parallel. \\
\intertext{It follows from~\eqref{eq:BD}, that $\overline{B}^\parallel=-B$ and hence:} 
\overline{\tilde{R}_{\s}^{\pi}}^\parallel &=(-B)^{|E(\s)|}\cdot D^{-\mathtt{O}_{\pi}+\mathtt{U}_{\pi}}\cdot\prod_{k=1}^{|V(\s)|-1}\frac{D^{-\delta^{k}_{\pi} \cdot \mathtt{C}^{k}_\pi}}{1-D^{-\mathtt{C}^{k}_\pi}}.\\
\intertext{We apply the change of variable $\overline{k}\coloneqq |V(\s)|-k$, and obtain, by \cref{fact:mirrorUniv}: } 
\overline{\tilde{R}_{\s}^{\pi}}^\parallel &=(-1)^{|E(\s)|}B^{|E(\s)|}\cdot D^{-\mathtt{U}^{\overline{\pi}}_\s+\mathtt{O}^{\overline{\pi}}_\s}\cdot\prod_{k=1}^{|V(\s)|-1}\frac{D^{(\delta_{\overline{\pi}}^{\overline{k}}-1) \cdot \mathtt{C}_{\overline{\pi}}^{\overline{k}}}}{1-D^{-\mathtt{C}_{\overline{\pi}}^{\overline{k}}}}
\\
\overline{\tilde{R}_{\s}^{\pi}}^\parallel &=(-1)^{|E(\s)|+|V(\s)|-1}\cdot \tilde{R}_{\s}^{\overline{\pi}}.
\end{align}
Since $\s$ is a unicellular map of an orientable surface, by Euler's formula, we get $(-1)^{|E(\s)|+|V(\s)|-1}=1$. Hence:
\begin{equation}
\overline{\tilde{R}_{\s}^{\pi}}^\parallel =\tilde{R}_{\s}^{\overline{\pi}}.
\end{equation}
From \eqref{eq:tbul} and \eqref{eq:tcir} of Section~\ref{sec:symRat}, it is easy to see that, as a function of $D$, $\bar{t}^\parallel =t$. Hence reincorporating the weights of the scheme stems in $\overline{\tilde{R}_{\s}^{\pi}}$ gives:
\begin{equation}
\overline{R_{\s}^{\pi}}^\parallel=R_{\s}^{\overline{\pi}},
\end{equation}
which concludes the proof.
\end{proof}

\begin{remark}\label{rem:diff2proofs}
Note that the proof presented here differs from the one given in \cite{Lep19}. In the former, the schemes are also grouped by their height order. However, unlike here, where the equality cases of the height ordering are resolved with the naming, to get a total ordering, the height ordering used in \cite{Lep19} was only a partial ordering. It results in some alternated sum on partial orders, that end up canceling out after a careful analysis of their structure.
\end{remark}

\section{The bivariate case: Combinatorial proof of Theorem~\ref{thm:MirrorR}}\label{sec:bivSym}

\subsection{Signed height orders}
In the bivariate case, the weight of a branch in a labeled scheme depends not only on the difference of its labels but also on their parity. 
For this reason, we introduce \emph{signed bijections}, that allow to group labeled schemes not only through the height ordering of their vertices, but also through the parity of the difference of heights between any pair of vertices. 

\smallskip

To do so, some additional terminology is needed. For $\s\in \mathcal{S}$, a \emph{signed bijection} of $\s$ is a couple $(\pi,\zeta)\in \mathfrak{S}_{\s}\times (\mathbb{Z}/2\mathbb{Z})^{[|V(\s)|-1]}$.
The set of signed bijections associated to $\s$ is denoted by $\mathfrak{S}_{\s}^{\ptdeux}$.

For $\l\in\mathcal{L}$, such that $\us(\l)=\s$, the \emph{signed height ordering of $\l$} is the couple made of $\pi_{\l}$ -- its height ordering -- and of the function:
\begin{equation}
\zeta_{\l}:[|V(\s)|-1]\to \mathbb{Z}/2\mathbb{Z} \quad \text{ such that }\quad \zeta_{\l}(k)\equiv h_\l(\pi_{\l}(k+1))-h_\l(\pi_{\l}(k)) \text{ mod }2.
\end{equation}

\bigskip

{\bf For the rest of this section, we fix $\s \in \mathcal{S}$, $\nu$ a consistent naming of $\s$ and $(\pi,\zeta) \in \mathfrak{S}_{\s}^{\ptdeux}$.
}

\bigskip

We denote by $\mathcal{L}_{\s}^{(\pi,\zeta)}$ the set of labeled schemes $\l$ such that $\us(\l)=\s$ and $(\pi_{\l},\zeta_{\l})=(\pi,\zeta)$. Observe for later use that, for any $\l\in \mathcal{L}_{\s}^{(\pi,\zeta)}$, the parity of the height of its vertices is prescribed by $\zeta$ (and the fact that the root vertex of $\l$ has height $0$ by Claim~\ref{cla:height0}). In particular, either $h_\l\big(\pi^{-1}(1)\big)$ is odd for every $\l\in \mathcal{L}_{\s}^{(\pi,\zeta)}$, or it is even for every $\l\in \mathcal{L}_{\s}^{(\pi,\zeta)}$.

We denote by $\mathcal{R}_{\s}^{(\pi,\zeta)}$ the subset of elements of $\mathcal{R}$ whose labeled scheme belongs to $\mathcal{L}_{\s}^{(\pi,\zeta)}$. Similarly to the previous section, we write $\tilde R_{\s}^{(\pi,\zeta)}$ for the generating series of elements of $\mathcal{R}_{\s}^{(\pi,\zeta)}$, in which we neglect the contribution of the scheme stems. By Lemma~\ref{lem:decompCore}, we have: 
\begin{equation}\label{eq:RtildeBivariate}
\tilde{R}_{\s}^{(\pi,\zeta)}=B^{|E(\s)|}\cdot S^{(\pi,\zeta)}_\s(D_\bullet,D_\circ),
\end{equation}
where $S^{(\pi,\zeta)}_\s$ is the generating series defined by: 
\begin{equation}\label{eq:defSFull}
    S_\s^{(\pi,\zeta)}(D_\bullet,D_\circ)=
    \sum_{\l\in \mathcal{L}_\s^{(\pi,\zeta)}}\left(
     \prod_{e\in \Eup(\l)}\Delta_{\lambda_\t^{0}(e)}^{\lambda_\t^{1}(e)}(D_{\bullet},D_\circ)
    \prod_{e\in \Edo(\l)}\Delta_{\lambda_\t^{1}(e)}^{\lambda_\t^{0}(e)}(D_\circ,D_{\bullet})\right).
\end{equation}
\medskip

To establish Theorem~\ref{thm:MirrorR}, we extend the \emph{mirror operation} of Section~\ref{sub:MainMirror} to signed height orders. For $(\pi,\zeta)\in \mathfrak{S}_{\s}^{\ptdeux}$ and $k\in [|V(\s)|-1]$ we set:
\begin{equation}\label{eq:defMirrorZeta}
    \overline \zeta(k):=\zeta(\bar k), \quad \text{where we recall that } \bar k= |V(\s)|-k,
\end{equation}
and define $\overline{(\pi,\zeta)}:=(\overline \pi,\overline \zeta)$. 
\medskip

Then, the main result of this section is the following: 
\begin{proposition}\label{prop:MirrorUltimeS} We have the following inequality of generating series:
\begin{equation}
\overline{\left(S_\s^{(\pi,\zeta)}(D_\bullet,D_\circ)+S_\s^{(\pi,\zeta)}(D_\circ,D_\bullet)\right)}^\parallel =(-1)^{|V(\s)|-1}\left(S_\s^{\overline{(\pi,\zeta)}}(D_\bullet,D_\circ)+S_\s^{\overline{(\pi,\zeta)}}(D_\circ,D_\bullet)\right)
\end{equation}
\end{proposition}
The rest of this section is devoted to the proof of this result and how this implies Theorem~\ref{thm:MirrorR}. The proof is rather technical to say the least, and to give at least some intuitions about the quantity at stakes before dwelving into the computations, we start with an introductory example.

\begin{example}
Consider the graph $\mathfrak{g}$ reduced to two vertices $u$ and $v$ linked by an edge, and rooted at the vertex $u$. Set $\pi(u)=1$ and $\pi(v)=2$, and $\nu(1)=u$ and $\nu(2)=v$. This graph does obviously not correspond to a scheme, but we are going to compute $S^{(\pi,\zeta)}_{\mathfrak{g}}$ as if $\mathfrak{g}$ were a scheme with a unique unshifted edge (so that the naming $\nu$ is necessarily consistent). Moreover, the edge $\{u,v\}$ is oriented from $v$ to $u$ in the canonical orientation of $\mathfrak{g}$ so that $\{u,v\}\in E^\downarrow(\mathfrak{g})$, with this choice of $\pi$.

We start with the case where $\zeta(1)=0$. Since $\pi(u)=1\leq 2=\pi(v)$, we have that $h(u)\leq h(v)$. The case of equality is possible since $\nu(1)=u$ and $\nu(2)=v$. It follows from the definition of $S^{(\pi,\zeta)}_{\mathfrak{g}}$ given in~\eqref{eq:defSFull} that: 
\[
S^{(\pi,\zeta)}_{\mathfrak{g}}(D_\bullet,D_\circ)=\sum_{k\geq 0} D_\circ^kD_\bullet^k = \frac{1}{1-D_\circ D_\bullet}.
\]
Similarly, we have:
\[
S^{\overline{(\pi,\zeta)}}_{\mathfrak{g}}(D_\bullet,D_\circ)=\sum_{k\geq 1} D_\circ^kD_\bullet^k = \frac{D_\circ D_\bullet}{1-D_\circ D_\bullet},
\]
from which it follows that $\overline{S^{(\pi,\zeta)}_{\mathfrak{g}}}^\parallel = \overline{S^{(\pi,\zeta)}_{\mathfrak{g}}}^\times= (-1)^{|V(\mathfrak{g})|-1}S^{\overline{(\pi,\zeta)}}_{\mathfrak{g}}$, similarly to the univariate case (where the distinction between $\parallel$-symmetric and $\times$-symmetric was not pertinent). 

We now deal with the case $\zeta(1)=1$, this time we have: 
\[
S^{(\pi,\zeta)}_{\mathfrak{g}}(D_\bullet,D_\circ)=D_\circ\cdot\sum_{k\geq 0} D_\circ^kD_\bullet^k = \frac{D_\circ}{1-D_\circ D_\bullet},
\]
and similarly: 
\[
S^{\overline{(\pi,\zeta)}}_{\mathfrak{g}}(D_\bullet,D_\circ)=D_\circ\cdot\sum_{k\geq 0} D_\circ^kD_\bullet^k = \frac{D_\circ}{1-D_\circ D_\bullet},
\]
from which it follows that $\overline{S^{(\pi,\zeta)}_{\mathfrak{g}}}^\times = (-1)^{|V(\mathfrak{g})|-1}S^{\overline{(\pi,\zeta)}}_{\mathfrak{g}}$. We insist on the fact that this equality does not hold for $\parallel$-symmetry. 

However for both values of $\zeta$, it follows directly from the previous computations that the result of Proposition~\ref{prop:MirrorUltimeS} holds for the graph $\mathfrak{g}$. 

\end{example}

\subsection{Truncations: definition and decomposition} 
%
Unlike in the univariate case developed in the previous section, we could not obtain a simple closed formula for the bivariate series $S_{\s}^{(\pi,\zeta)}$. Instead of a bivariate analogue of \cref{lem:uniEnumOrdered}, we give a recursive description of that series, which will be establish by induction on the number of vertices in the scheme. 

\begin{figure}[t]
\centering
    \subcaptionbox{Orientation of $\us(\l)$.\label{subfig:orientationTrunc}}
    {\includegraphics[page=9,scale=.85]{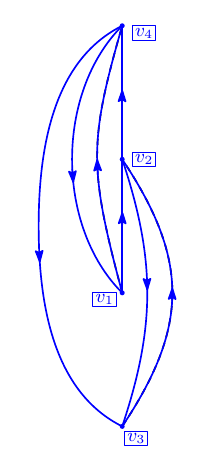}}\qquad\ 
      \subcaptionbox{Ascending truncation at height 1,\label{subfig:trunc1}}
    {\includegraphics[page=10,scale=.85]{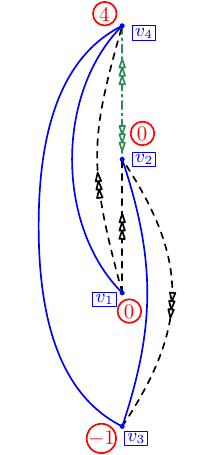}}\qquad\ 
     \subcaptionbox{at height 2,\label{subfig:trunc2}}
    {\includegraphics[page=11,scale=.85]{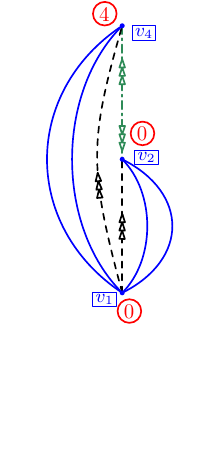}}\qquad\ 
     \subcaptionbox{and at height 3\label{subfig:trunc3}}
    {\includegraphics[page=12,scale=.85]{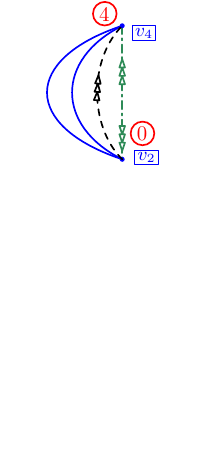}}
\caption{\label{fig:Truncation} We still consider the labeled scheme of Figure~\ref{fig:naming}\subref{subfig:schemeNaming}, with the same naming $\nu$. In~\subref{subfig:orientationTrunc}, we represent \emph{as a graph} the interior map of its unlabeled scheme. The ordering (from bottom to top) of its vertices is given by its height ordering relative to $\nu$. The orientation of its edges corresponds to the canonical orientation of $\l$.\\
In~\subref{subfig:trunc1}, \subref{subfig:trunc2} and \subref{subfig:trunc3}, we represent respectively the ascending truncation of $\l$ at height 1, 2 and 3. For sake of clarity, we do not represent the projection of the canonical orientation of $\l$ on its truncations. But, it can be deduced from \subref{subfig:orientationTrunc}.}
\end{figure}

\begin{definition}
For $\l \in\mathcal{L}_{\s}^\pi$ and $k\in [|V(\l)|]$, the \emph{ascending truncation of $\l$ at rank $k$} -- denoted $\Pi^{k:+}(\l)$ -- is the oriented and labeled graph, with typed half-edges, obtained as follows, see Figure~\ref{fig:Truncation}. 

Start with $\l$ endowed with its canonical orientation, and label its vertices by their height. 
Consider the subset $X$ of vertices of $\l$ defined by $X=\pi_{\l}([k-1])$, then:
\begin{itemize}
    \item Delete all the stems of $\l$. 
    \item Declare that all half-edges incident to vertices of $X$ are unshifted (i.e: replace offset edges towards a vertex of $X$ and shifted edges between two vertices of $X$ by unshifted edges, and shifted edges between $u\in X$ and $v\notin X$ by an offset edge towards $v$). In other words, consider that all the relative labels around vertices of $X$ are now equal to 0. 
    \item Finally, merge all the vertices of $X$ with $\pi_{\l}(k)$, which keeps its label $h_{\l}(\pi_{\l}(k))$, and erase the self-loops thus created.
\end{itemize}

We define similarly the \emph{descending truncation of $\l$ at rank $k$} -- denoted by $\Pi^{k:-}(\l)$ -- by setting instead $X=\pi_{\l}([k,|V(\s)|])$ in the previous construction. 
\end{definition}
We now introduce the set of graphs that can be obtained as the truncation of a labeled scheme. 
\begin{definition}
For $(\pi,\zeta)\in \mathfrak{S}_{\s}^{\ptdeux}$, $k\in |V(\s)|$, and $\pm\in\{-,+\}$ we denote by:
\[
\mathcal{T}_{(\pi,\zeta)}^{k:\pm}=\left\{\Pi^{k,\pm}(\l)\,|\,\l\in \mathcal{L}_{\s}^{(\pi,\zeta)}\right\}.
\]

Additionally, we denote by $\mathcal{T}_{(\pi,\zeta)}$ the set of all truncations, i.e. 
\[
    \mathcal{T}_{(\pi,\zeta)}=\bigcup_{k\in |V(\s)|,\pm\in \{-,+\}}\mathcal{T}_{(\pi,\zeta)}^{k:\pm}.
\]
\end{definition}

Vertices of a truncation inherit naturally the height of the vertices of their preimage. For $\t\in\mathcal{T}_{(\pi,\zeta)}$ and $v\in V(\t)$, we denote by $h_\t(v)$ this value. 
We now extend the definition of $\lamZ(e)$ and $\lamU(e)$ for $e\in E(\t)$. Informally, corners incident to vertices of rank not smaller than $k$ in $\l$ keep their labels, while other corners are labeled $h_\l(\pi(k))$. More formally, for $e=\{u_0,u_1\}\in E(\l)$ oriented from $u_0$ to $u_1$, for $k\in[|V(\l)|]$, and for $\varepsilon\in\{0,1\}$, we define:
\begin{align}
    \lambda_\l^{\varepsilon,k:+}(e)&=
    \begin{cases}
    h_\l(\pi(k)), &\text{ if }\pi^{-1}(u_\varepsilon)<k\text{ or  if } \pi^{-1}(u_\varepsilon)=k\text{ and } \pi^{-1}(u_{1-\varepsilon})<k,\\
    \lambda_\l^\varepsilon(e), &\text{otherwise.}
    \end{cases}\\
\intertext{and}
    \lambda_\l^{\varepsilon,k:-}(e)&=
    \begin{cases}
    h_\l(\pi(k)), &\text{ if }\pi^{-1}(u_\varepsilon)>k\text{ or  if } \pi^{-1}(u_\varepsilon)=k\text{ and } \pi^{-1}(u_{1-\varepsilon})>k,\\
    \lambda_\l^\varepsilon(e), &\text{otherwise.}
    \end{cases}
\end{align}

For $k\in [|V(\s)|]$, let $\l_1,\l_2 \in \mathcal{L}_{\s}^{(\pi,\zeta)}$ be such that $\Pi^{k:+}(\l_1)=\Pi^{k:+}(\l_2)$, and denote $\t$ their common truncation. Note that, for any edge $e$, $\lambda_{\l_1}^{\varepsilon,k:+}(e)=\lambda_{\l_2}^{\varepsilon,k:+}(e)$. We then define  $\lambda_{\t}^{\varepsilon}(e)$ as this common value. Similarly, if $\t=\Pi^{k:-}(\l_1)=\Pi^{k:-}(\l_2)$, we set $\lambda_{\t}^{\varepsilon}(e):=\lambda_{\l_1}^{\varepsilon,k:-}(e)=\lambda_{\l_2}^{\varepsilon,k:-}(e)$.

\smallskip

The definition of $\Eup$ and $\Edo$ first given in~Section~\ref{sec:brancheMotz} can then be naturally extended to truncations and for $\t\in \mathcal{T}_{\s}^{(\pi,\zeta)}$, we set:
\begin{equation}
\Eup(\t)=\Big\{e\in E(\t)\,|\,\lambda_\t^{0}(e)\leq \lambda_\t^{1}(e)\Big\}\quad \text{ and }\quad \Edo(\t)=\Big\{e\in E(\t)\,|\,\lambda_\t^{0}(e)> \lambda_\t^{1}(e)\Big\}.
\end{equation}

Finally for $\pm \in \{-,+\}$, we define the following family of generating series: 
\begin{equation}\label{eq:defSGen}
    S_{(\pi,\zeta)}^{k:\pm}(D_\bullet,D_\circ)=
    \sum_{\t\in \mathcal{T}_{(\pi,\zeta)}^{k:\pm}}\left(
     \prod_{e\in \Eup(\t)}\Delta_{\lambda_\t^{0}(e)}^{\lambda_\t^{1}(e)}(D_{\bullet},D_\circ)
    \prod_{e\in \Edo(\t)}\Delta_{\lambda_\t^{1}(e)}^{\lambda_\t^{0}(e)}(D_\circ,D_{\bullet})\right).
\end{equation}

Observe that $S_{(\pi,\zeta)}^{1:+}(D_\bullet,D_\circ)=S_{(\pi,\zeta)}^{|V(s)|:-}=S_\s^{(\pi,\zeta)}$, so that Proposition~\ref{prop:MirrorUltimeS} is in fact a particular case of the following result: 
\begin{proposition}\label{prop:MirrorTruncationUltime}
For any $k\in [|V(\s)|]$, we have: 
\begin{equation}\overline{\left(S^{k:+}_{(\pi,\zeta)}(D_\bullet,D_\circ)+S^{k:+}_{(\pi,\zeta)}(D_\circ,D_\bullet)\right)}^\parallel =(-1)^{|V(\s)|-1}\left(S^{\bar k+1:-}_{\overline{(\pi,\zeta)}}(D_\bullet,D_\circ)+S^{\bar k+1:-}_{\overline{(\pi,\zeta)}}(D_\circ,D_\bullet)\right)
\end{equation}\end{proposition}

This proposition will be proven in Section~\ref{sub:mirrorTrunc} and its proof relies on the following recursive decomposition of $S_{(\pi,\zeta)}^{k:\pm}$:

\begin{lemma}\label{lem:bivEnumOrdered}
The family of functions $\Big(S_{(\pi,\zeta)}^{k:\pm},\text{ with }\pm\in\{-,+\},k\in [|V(\s)|]\Big)$ admits the following recursive characterization:
\begin{equation}
S_{(\pi,\zeta)}^{|V(\s)|:+}=1, \quad \text{ and }\quad
S_{(\pi,\zeta)}^{1:-}=1.
\end{equation}
Fix $k\in [|V(\s)-1|]$ and define $c_{(\pi,\zeta)}^k\in\{\bullet,\circ\}$ by:
\begin{equation}
c_{(\pi,\zeta)}^k:=\begin{cases}
\bullet &\text{if } h(\pi(k)) \text{ is even}\\
\circ & \text{otherwise.} 
\end{cases}
\end{equation}
Then, writing $c=c_{(\pi,\zeta)}^k$ and defining $-c$ as the opposite color, we have:
\begin{equation}
S_{(\pi,\zeta)}^{k:+}=\begin{cases}
\dfrac{(D_{c})^{\mathtt{O}_\pi^{k,\uparrow}}(D_{c})^{\mathtt{O}_\pi^{k,\downarrow}}}{(D_{c})^{U_\pi^{k,\uparrow}}(D_{c})^{U_\pi^{k,\downarrow}}}\cdot\dfrac{(D_\circ D_\bullet)^{\delta^{k}_{\pi}\cdot \mathtt{C}_{\pi}^{k}}}{1-(D_\circ D_\bullet)^{\mathtt{C}_{\pi}^{k}}}\cdot
S_{(\pi,\zeta)}^{k+1:+} & \text{if }\zeta(k)=0,\\
~&\\
\dfrac{(D_{-c})^{\mathtt{O}_\pi^{k,\uparrow}}(D_{c})^{\mathtt{O}_\pi^{k,\downarrow}}}{(D_{c})^{U_\pi^{k,\uparrow}}(D_{-c})^{U_\pi^{k,\uparrow}}}\cdot\dfrac{(D_c)^{\mathtt{C}_\pi^{k,\uparrow}}\cdot (D_{-c})^{\mathtt{C}_\pi^{k,\downarrow}}}{1-(D_\circ D_\bullet)^{\mathtt{C}_{\pi}^{k}}} 
\cdot
S_{(\pi,\zeta)}^{k+1:+} 
& \text{if }\zeta(k)=1.
\end{cases}
\end{equation}
and similarly: 
\begin{equation}
S_{(\pi,\zeta)}^{k+1:-}=\begin{cases}
\dfrac{(D_{c})^{\mathtt{O}_\pi^{k,\uparrow}}(D_{c})^{\mathtt{O}_\pi^{k,\downarrow}}}{(D_{c})^{U_\pi^{k,\uparrow}}(D_{c})^{U_\pi^{k,\downarrow}}}\cdot\dfrac{(D_\circ D_\bullet)^{\delta^{k}_{\pi}\cdot \mathtt{C}_{\pi}^{k}}}{1-(D_\circ D_\bullet)^{\mathtt{C}_{\pi}^{k}}}\cdot
S_{(\pi,\zeta)}^{k:-} & \text{if }\zeta(k)=0,\\
~&\\
\dfrac{(D_{-c})^{\mathtt{O}_\pi^{k,\uparrow}}(D_{c})^{\mathtt{O}_\pi^{k,\downarrow}}}{(D_{c})^{U_\pi^{k,\uparrow}}(D_{-c})^{U_\pi^{k,\uparrow}}}\cdot\dfrac{(D_c)^{\mathtt{C}_\pi^{k,\uparrow}}\cdot (D_{-c})^{\mathtt{C}_\pi^{k,\downarrow}}}{1-(D_\circ D_\bullet)^{\mathtt{C}_{\pi}^{k}}} 
\cdot
S_{(\pi,\zeta)}^{k:-} 
& \text{if }\zeta(k)=1.
\end{cases}
\end{equation}
The quantities $\mathtt{C}_{\pi}^{k,\uparrow}$ and $\mathtt{C}_{\pi}^{k,\downarrow}$ are introduced in Definition~\ref{def:Ek}, and the quantities $\mathtt{O}_\pi^{k,\uparrow}$, $\mathtt{O}_\pi^{k,\downarrow}$, $\mathtt{U}_\pi^{k,\uparrow}$ and $\mathtt{U}_\pi^{k,\downarrow}$ in Definition~\ref{def:OkUk} in the course of the proof.

\end{lemma}

\smallskip

\begin{proof}
From the definition of ascending and descending truncation, the sets $\mathcal{T}_{(\pi,\zeta)}^{|V(\s)|:+}$ and $\mathcal{T}_{(\pi,\zeta)}^{1:-}$ are all reduced to a single element with one vertex and no edge. 
The initialization cases follow easily.
\smallskip

For any $e\in E(\t)$, we investigate its contribution to $S_{(\pi,\zeta)}^{k:+}$. If both extremities of $e$ have rank greater than $k$, then 
the contribution of $e$ to $S_{(\pi,\zeta)}^{k:+}$ is the same as its contribution to $S_{(\pi,\zeta)}^{k+1:+}$. So that, only edges contributing to $\mathtt{C}_\pi^k$ (see Definition~\ref{def:Ck}) need to be studied. We need the following refinement of $\mathtt{C}_\pi^k$:

\begin{definition}\label{def:Ek}
For any $k\in |V(\s)|$, we partition the set of edges of $\s$ contributing to $\mathtt{C}_\pi^k$ into two sets, depending on whether they are increasing or decreasing. More precisely, for $e=\{u,v\}\in E(\s)$ oriented from $u$ to $v$, we set: 
\begin{equation}
\begin{cases}
e\in E_{\pi}^{k,\uparrow} & \text{ if }\pi^{-1}(u)\leq k<\pi^{-1}(v)\\
e\in E_{\pi}^{k,\downarrow} & \text{if }\pi^{-1}(v)\leq k<\pi^{-1}(u).
\end{cases}
\end{equation}
By a slight abuse of notation, we identify elements of $E_{\pi}^{k,\uparrow}$ and of $E_{\pi}^{k,\downarrow}$ with their image in $\Pi^{k:+}(\s)$ or in $\Pi^{k+1:-(\s)}$, and define $\mathtt{C}_{\pi}^{k,\uparrow}:=|E_{\pi}^{k,\uparrow}|$ and $\mathtt{C}_{\pi}^{k,\downarrow}:=|E_{\pi}^{k,\downarrow}|$.
\end{definition}

We start with the case where $e\in E_{\pi}^{k,\uparrow}$. Its contribution can be decomposed into $w_1(e):=\Delta_{\lambda_t^0(e)}^{h_\t(\pi(k+1))}(D_\bullet,D_\circ)$\footnote{We slightly abuse notation here since $\Delta_i^j$ is only defined for $i\leq j$, and we might have $\lambda_t^0(e)=h_\t(\pi(k+1))+1$, in that case we set $w_1(e):=1/D_{c}$, which corresponds to $\zeta_k=0$ and $J_k=0$ in~\eqref{eq:w1Shifted}.} and $w_2(e):=\Delta_{h_\t(\pi(k+1))}^{\lambda_t^1(e)}(D_\bullet,D_\circ)$. We first compute $w_1(e)$. Set:
\[
J_k\coloneqq h_\l\Big(\pi(k+1)\Big)-h_\l\Big(\pi(k)\Big). 
\]
If $\zeta(k)=1$, $J_k$ can be any positive odd integer. Whereas, if $\zeta(k)=0$, $J_k$ can be any even nonnegative integer if $\nu(\pi(k))<\nu(\pi(k+1))$ or any even integer greater than 1 otherwise (equivalently, $J_k$ is an even integer not smaller than $\delta^{k}_{\pi}$, as defined in~\eqref{eq:defDelta}).

If the half-edge of $e$ incident to $\pi(k)$ is unshifted, i.e. if $\lamZ_\t(e)=h_\t(\pi(k))$, we have:
\begin{equation}
w_1(e)=
\begin{cases}(D_cD_{-c})^{J_k/2}&\text{ if }\zeta_k=0\\
D_{c}(D_cD_{-c})^{(J_k-1)/2}&\text{ if }\zeta_k=1.
\end{cases}
\end{equation}
Otherwise, i.e. if $\lamZ_\t(e)=h_\t(\pi(k))+1$, we need to divide the above formula by $D_{c}$. In that case, we obtain: 
\begin{equation}\label{eq:w1Shifted}
w_1(e)=
\begin{cases}
D_{-c}(D_cD_{-c})^{J_k/2-1}&\text{ if }\zeta_k=0\\
(D_cD_{-c})^{(J_k-1)/2}&\text{ if }\zeta_k=1.
\end{cases}
\end{equation}
We now move to $w_2(e)$. If $e$ is incident to $\pi(k+1)$, then $w_2(e)=1$ if $\lambda_\t^1(e)=h_\t(\pi(k+1))$,    and otherwise is equal to: 
\[
    w_2(e)= \begin{cases}
    D_c&\text{ if }\zeta(k)=0,\\
    D_{-c}&\text{ if }\zeta(k)=1.
    \end{cases}
\] 
In this case, the image of $e$ in $\Pi^{k+1:+}(\t)$ is a self-loop and is deleted by definition of the truncation. Otherwise, $w_2(e)$ is exactly equal to the contribution to $S_{(\pi,\zeta)}^{k+1:+}$ of the image of $e$ in $\Pi^{k+1:+}(\t)$.
\smallskip

If $e\in  E_{\pi}^{k,\downarrow}$, the exact same analysis can be carried out and we obtain the same contributions up to replacing $c$ by $-c$ everywhere. 

Recall the definitions of $E_\pi^{k,\uparrow}$ and $E_\pi^{k,\downarrow}$ and of $\mathtt{O}_{\pi}$ and $\mathtt{U}_{\pi}$ given respectively in Definition~\ref{def:Ek} and in Definition~\ref{def:OkUk}. We need to introduce the following refinement of $\mathtt{O}_\pi^{k}$ and $\mathtt{U}_\pi^{k}$:
\begin{definition}\label{def:OkUk}
For $k\in [|V(\s)|-1]$, define respectively $\mathtt{O}_\pi^{k,\uparrow}$ and $\mathtt{U}_\pi^{k,\uparrow}$ as the number of edges $\{u,v\}$ in $E_\pi^{k,\uparrow}$ such that the half-edge incident to $u$ is shifted, and such that $\pi(v)<\pi(u)=k+1$ (respectively $k=\pi(u)<\pi(v)$).\\
We define similarly $\mathtt{O}_\pi^{k,\downarrow}$ and $\mathtt{U}_\pi^{k,\downarrow}$.
\end{definition}
Summarizing the contributions of all edges in $\t$, and summing over $\t$, and over all possible values of $J_k$ we get: 

\noindent For $\zeta(k)=1$: 
\begin{align}
\frac{S_{(\pi,\zeta)}^{k:+}}{S_{(\pi,\zeta)}^{(k+1):+}}
&= 
     \sum_{i\geq 0}\left(
     \frac{(D_{-c})^{\mathtt{O}_\pi^{k,\uparrow}}}{(D_{c})^{U_\pi^{k,\uparrow}}}\left(\prod_{e\in  E_{\pi}^{k,\uparrow}}
     D_{c}(D_\circ D_{\bullet})^i\right)\cdot
     \frac{(D_{c})^{\mathtt{O}_\pi^{k,\downarrow}}}{(D_{-c})^{U_\pi^{k,\downarrow}}}\left(\prod_{e\in  E_{\pi}^{k,\downarrow}}
     (D_{-c})(D_\circ D_{\bullet})^i\right)\right)\\
     &=\frac{(D_{-c})^{\mathtt{O}_\pi^{k,\uparrow}}(D_{c})^{\mathtt{O}_\pi^{k,\downarrow}}}{(D_{c})^{U_\pi^{k,\uparrow}}(D_{-c})^{U_\pi^{k,\uparrow}}}\cdot\frac{(D_c)^{\mathtt{C}_\pi^{k,\uparrow}}\cdot (D_{-c})^{\mathtt{C}_\pi^{k,\downarrow}}}{1-(D_\circ D_\bullet)^{\mathtt{C}_{\pi}^{k}}} 
\end{align}

For $\zeta(k)=0$, recall that $J_k$ can take any even value not smaller than $\delta^{k}_{\pi}$. Hence, we get:
\begin{equation}
\frac{S_{(\pi,\zeta)}^{k:+}}{S_{(\pi,\zeta)}^{(k+1):+}}
=\frac{(D_{c})^{\mathtt{O}_\pi^{k,\uparrow}}(D_{c})^{\mathtt{O}_\pi^{k,\downarrow}}}{(D_{c})^{U_\pi^{k,\uparrow}}(D_{c})^{U_\pi^{k,\uparrow}}}\cdot\frac{(D_\circ D_\bullet)^{\delta^{k}_{\pi}\cdot \mathtt{C}_{\pi}^{k}}}{1-(D_\circ D_\bullet)^{\mathtt{C}_{\pi}^{k}}} 
\end{equation}

The other recurrence relation (for descending truncations) is obtained similarly.
\end{proof}

\subsection{Proof of Propositions~\ref{prop:MirrorUltimeS} and~\ref{prop:MirrorTruncationUltime} via truncations}\label{sub:mirrorTrunc}
\begin{lemma}\label{lem:bivSymMir}
Let $(\pi,\zeta)\in \mathfrak{S}_{\s}^{\ptdeux}$, and $k\in [|V(\s)|-1]$. Then, we have: 
\begin{equation}
\overline{
\left(
S_{(\pi,\zeta)}^{k:+}\Big/
S_{(\pi,\zeta)}^{k+1:+}
\right)}^\parallel=
-
\left(S_{\overline{(\pi,\zeta)}}^{\overline{k}+1:-}\Big/S_{\overline{(\pi,\zeta)}}^{\overline{k}:-}\right)\quad\text{if }\zeta(k)=0,
\end{equation}
and
\begin{equation}
\overline{
\left(
S_{(\pi,\zeta)}^{k:+}\Big/
S_{(\pi,\zeta)}^{k+1:+}
\right)}^\times=
-
\left(S_{\overline{(\pi,\zeta)}}^{\overline{k}+1:-}\Big/S_{\overline{(\pi,\zeta)}}^{\overline{k}:-}\right) \quad\text{if }\zeta(k)=1.
\end{equation}
\end{lemma}

Before proving this lemma, let us first comment on why it concludes the proof of Proposition~\ref{prop:MirrorTruncationUltime} and hence of Proposition~\ref{prop:MirrorUltimeS}. 

We proceed by induction. The result is clear for $k=|V(\s)|$. Indeed, by Lemma~\ref{lem:bivEnumOrdered} we have: 
\begin{equation}
S_{(\pi,\zeta)}^{|V(\s)|:+}=1=
S_{(\pi,\zeta)}^{1:-}.
\end{equation}
Then, recall that for any function $f(x,y)$ of two variables, we write $f^{\yy}:=f(x,y)+f(y,x)$ for the symmetrization of $f$. It follows from Lemma~\ref{lem:bivSymMir} that for any $k\in [|V(\s)|-1]$, we have:
\begin{equation}
\overline{
\left(
S_{(\pi,\zeta)}^{k:+}\Big/
S_{(\pi,\zeta)}^{k+1:+}
\right)^{\yy}}^\parallel=-
\left(S_{\overline{(\pi,\zeta)}}^{\overline{k}+1:-}\Big/S_{\overline{(\pi,\zeta)}}^{\overline{k}:-}\right)^{\yy},
\end{equation}
which concludes the proof of Proposition~\ref{prop:MirrorTruncationUltime}.

The key ingredient of the proof of Lemma~\ref{lem:bivSymMir} is the following extension and refinement of Fact~\ref{fact:mirrorUniv}, which follows directly from Definitions~\ref{def:Ek} and~\ref{def:OkUk}:
\begin{fact}\label{fact:mirroBiv}
Let $(\pi,\zeta)\in \mathfrak{S}_{\s}^{\ptdeux}$, and $k\in [|V(\s)|-1]$. Then, we have: 
\begin{equation}
\mathtt{O}_{\pi}^{k,\uparrow}=\mathtt{U}_{\overline{\pi}}^{\bar k,\downarrow},\quad\text{and}\quad\mathtt{O}_{\pi}^{k,\downarrow}=\mathtt{U}_{\overline{\pi}}^{\bar k,\uparrow}
\end{equation}
and
\begin{equation}
\mathtt{C}_{\pi}^{k,\uparrow}=\mathtt{C}_{\overline{\pi}}^{\bar k,\downarrow}\quad\text{and}\quad\mathtt{C}_{\pi}^{k,\downarrow}=\mathtt{C}_{\overline{\pi}}^{\bar k,\uparrow}.
\end{equation}
\end{fact}

\begin{proof}[Proof of Lemma~\ref{lem:bivSymMir}]
We only treat the case where $\zeta(k)=1$, the case $\zeta(k)=0$ being similar and closer to the univariate setting.
As in Lemma~\ref{lem:bivEnumOrdered}, set $c=c_{(\pi,\zeta)}^k$, then we can write: 
\begin{align}
\overline{
\left(
S_{(\pi,\zeta)}^{k:+}\Big/
S_{(\pi,\zeta)}^{k+1:+}
\right)}^\times
&=\overline{
\left(
\dfrac{
(D_{-c})^{\mathtt{O}_\pi^{k,\uparrow}}
(D_{c})^{\mathtt{O}_\pi^{k,\downarrow}}
}{
(D_{c})^{\mathtt{U}_\pi^{k,\uparrow}}
(D_{-c})^{\mathtt{U}_\pi^{k,\downarrow}}}
\cdot\dfrac{(D_c)^{\mathtt{C}_\pi^{k,\uparrow}}\cdot (D_{-c})^{\mathtt{C}_\pi^{k,\downarrow}}}{1-(D_\circ D_\bullet)^{\mathtt{C}_{\pi}^{k}}} 
\right)}^\times\\
&=\dfrac{
(D_{-c})^{\mathtt{U}_\pi^{k,\uparrow}}
(D_{c})^{\mathtt{U}_\pi^{k,\downarrow}}
}
{
(D_{c})^{\mathtt{O}_\pi^{k,\uparrow}}
(D_{-c})^{\mathtt{O}_\pi^{k,\downarrow}}
}
\cdot\dfrac{(D_{-c})^{-\mathtt{C}_\pi^{k,\uparrow}}\cdot (D_{c})^{-\mathtt{C}_\pi^{k,\downarrow}}}{1-(D_\circ D_\bullet)^{-\mathtt{C}_{\pi}^{k}}} 
\intertext{Applying Fact~\ref{fact:mirroBiv} and observing that $\mathtt{C}_\pi^{k,\uparrow}+\mathtt{C}_\pi^{k,\downarrow}=\mathtt{C}_\pi^{k}$, we can rewrite:}
\overline{
\left(
S_{(\pi,\zeta)}^{k:+}\Big/
S_{(\pi,\zeta)}^{k+1:+}
\right)}^\times
&=
\dfrac{
(D_{-c})^{\mathtt{O}_{\bar\pi}^{\bar k,\downarrow}}
(D_{c})^{\mathtt{O}_{\bar\pi}^{\bar k,\uparrow}}
}{
(D_{c})^{\mathtt{U}_{\bar\pi}^{\bar k,\downarrow}}
(D_{-c})^{\mathtt{U}_{\bar\pi}^{\bar k,\uparrow}}}
\cdot\dfrac{(D_c)^{\mathtt{C}_{\bar\pi}^{\bar k,\downarrow}}\cdot (D_{-c})^{\mathtt{C}_{\bar\pi}^{\bar k,\uparrow}}}{(D_\circ D_\bullet)^{\mathtt{C}_{{\bar\pi}}^{\bar k}}-1}.
\end{align}
Since $\zeta(k)=1$, $h(\pi^{-1}(k))$ and $h(\bar \pi^{-1}(\bar k))$ have opposite parity, so that $c^{\bar k}_{\overline{(\pi,\zeta)}}=-c$. We can then apply Lemma~\ref{lem:bivSymMir} to get a recurrence  relation between $S_{\overline{(\pi,\zeta)}}^{\overline{k}+1:-}$ and $S_{\overline{(\pi,\zeta)}}^{\overline{k}:-}$, which yields the desired result. 
\end{proof} 

\subsection{Proof of Theorem~\ref{thm:MirrorR}}\label{sub:proofMirrorR}

By \cref{lem:decompCore}, we can write:
\begin{equation}\label{eq:Rfinal}
    R_{\s}^{(\pi,\zeta)}(D_\bullet,D_\circ)= B^{|E(\s)|} \cdot S_\s^{(\pi,\zeta)}\cdot \sum_{\l\in \mathcal{L}_\s^{(\pi,\zeta)}}\prod_{r\in \bar S(\l)}\Delta_{\lambda_\l(r)}^{\lambda_\l(r)+1}(t_\bullet,t_\circ),
\end{equation}
where we recall that $\bar S(\l)$ denotes the set of rootable scheme stems of $\l$. It follows from~Lemma~\ref{lem:decompCore}, that a rootable stem $r$ contributes to a weight 
$t_{\bullet}$ if $h_\l(\v(r))+\lambda_\s(r)$ is even, and $t_{\circ}$ otherwise.
Then the total contribution of the rootable stems in $R_{\s}^{(\pi,\zeta)}$ is given by: 
\begin{equation}
\mathtt{P}_{(\pi,\zeta)}:=\sum_{\l \in \mathcal{L}_{\s}^{(\pi,\zeta)}}\prod_{r\in \bar S(\l)}
\begin{cases}
t_\bullet,&\text{ if }h_\l(\pi^{-1}(1))+\sum_{i=1}^{\pi^{-1}(\v(r)-1)}\zeta(i) + \lambda_\s(r) \text{ is even}, \\
t_\circ,&\text{ otherwise.}
\end{cases}
\end{equation}
For any $\bar \l \in \mathcal{L}_{\s}^{\overline{(\pi,\zeta)}}$, the parity of the height of vertices in $\bar \l$ are the same as in any $\l \in \mathcal{L}_{\s}^{(\pi,\zeta)}$, more formally: $h_\l(\pi(k))$ and $h_{\bar\l}(\bar\pi(\bar k+1))$ have the same parity for any $k\in [|V(\s)|]$. Hence, we have: $\mathtt{P}_{(\pi,\zeta)}=\mathtt{P}_{\overline{(\pi,\zeta)}}$. By~\eqref{eq:tbul} and~\eqref{eq:tcir}, $t_\circ$ and $t_\bullet$ are $\times$-symmetric, so that we also have:
\begin{equation}
\overline{\mathtt{P}_{(\pi,\zeta)}}^\times=\mathtt{P}_{\overline{(\pi,\zeta)}}, \quad \text{ and hence }\quad\overline{\mathtt{P}_{(\pi,\zeta)}^\yy}^\times=\mathtt{P}_{\overline{(\pi,\zeta)}}^\yy.
\end{equation}
Next, recall the expression of $B$ given in~\eqref{eq:BD}, from which it follows that: 
\[
\overline{B(D_\bullet,D_{\circ})}^\parallel=\overline{B(D_\bullet,D_{\circ})}^\times=-B(D_\circ,D_{\bullet}).
\]
Combined with Proposition~\ref{prop:MirrorUltimeS}, this gives: 
\begin{equation}
\overline{\left(R_{\s}^{(\pi,\zeta)}\right)^{\yy}}^\parallel =(-1)^{|E(\s)|+|V(\s)|-1}\cdot \left(R_{\s}^{\overline{(\pi,\zeta)}}\right)^{\yy}.
\end{equation}
By Euler's formula, for any unlabeled scheme, we have $|E(\s)|+|V(\s)|-1=2g$. And, since 
\[
    \left(R_{\s}^{\pi}\right)^{\yy}=\sum_{\zeta:[|V(\s)|-1]\rightarrow \mathbb{Z}/2\mathbb{Z}}\left(R_{\s}^{(\pi,\zeta)}\right)^{\yy},
\]
this concludes the proof.

\normalsize

\bibliographystyle{amsplain}
\bibliography{faceSommet}

\providecommand{\bysame}{\leavevmode\hbox to3em{\hrulefill}\thinspace}
\providecommand{\MR}{\relax\ifhmode\unskip\space\fi MR }
\providecommand{\MRhref}[2]{%
  \href{http://www.ams.org/mathscinet-getitem?mr=#1}{#2}
}
\providecommand{\href}[2]{#2}
\begin{thebibliography}{10}

\bibitem{Arq85}
D.~Arques, \emph{Une relation fonctionnelle nouvelle sur les cartes planaires
  point{\'e}es}, J. Combin. Theory Ser. B \textbf{39} (1985), no.~1, 27--42.

\bibitem{Arq87}
D.~Arqu{\`e}s, \emph{Relations fonctionnelles et d{\'e}nombrement des cartes
  point{\'e}es sur le tore}, J. Combin. Theory Ser. B \textbf{43} (1987),
  no.~3, 253--274.

\bibitem{ArqGio99}
D.~Arques and A.~Giorgetti, \emph{{\'E}num{\'e}ration des cartes point{\'e}es
  sur une surface orientable de genre quelconque en fonction des nombres de
  sommets et de faces}, J. Combin. Theory Ser. B \textbf{77} (1999), no.~1,
  1--24.

\bibitem{BenCan86}
E.A. Bender and E.R. Canfield, \emph{The asymptotic number of rooted maps on a
  surface}, J.~Combin.~Theory~Ser.~A \textbf{43} (1986), no.~2, 244--257.

\bibitem{BenCan91}
\bysame, \emph{The number of rooted maps on an orientable surface}, J. Combin.
  Theory Ser. B \textbf{53} (1991), no.~2, 293--299.

\bibitem{BeCaRi93}
E.A. Bender, E.R. Canfield, and L.B Richmond, \emph{The asymptotic number of
  rooted maps on a surface.: {II}. {E}numeration by vertices and faces}, J.
  Combin. Theory Ser. A \textbf{63} (1993), no.~2, 318--329.

\bibitem{BonzomChapuyDolega}
Valentin Bonzom, Guillaume Chapuy, and Maciej Do{\l}{\k{e}}ga,
  \emph{Enumeration of non-oriented maps via integrability}, Algebraic
  Combinatorics \textbf{5} (2022), no.~6, 1363--1390.

\bibitem{BDG04}
J.~Bouttier, {Ph.} {Di Francesco}, and E.~Guitter, \emph{Planar maps as labeled
  mobiles}, Electron. J. Combin \textbf{11} (2004), no.~1, R69.

\bibitem{CarCha15}
S.R. Carrell and G.~Chapuy, \emph{Simple recurrence formulas to count maps on
  orientable surfaces}, J.~Combin.~Theory~Ser.~A \textbf{133} (2015), 58--75.

\bibitem{ChMaSc09}
G.~Chapuy, M.~Marcus, and G.~Schaeffer, \emph{A bijection for rooted maps on
  orientable surfaces}, SIAM J. Discrete Math. \textbf{23} (2009), no.~3,
  1587--1611.

\bibitem{CorVau81}
R.~Cori and B.~Vauquelin, \emph{Planar maps are well labeled trees}, Canad. J.
  Math. \textbf{33} (1981), no.~5, 1023--1042.

\bibitem{DolegaLepoutre}
Maciej Do{\l}{\k{e}}ga and Mathias Lepoutre, \emph{Blossoming bijection for
  bipartite pointed maps and parametric rationality of general maps of any
  surface}, Adv. Appl. Math. \textbf{141} (2022), 102408.

\bibitem{Eyna16}
B.~Eynard, \emph{Counting surfaces}, vol.~70, Springer, 2011.

\bibitem{FlaSed}
{Ph}. Flajolet and R.~Sedgewick, \emph{Analytic combinatorics}, Cambridge
  University press, 2009.

\bibitem{LanZvo04}
S.K. Lando and A.K. Zvonkin, \emph{{Graphs on surfaces and their applications.
  Appendix by Don B. Zagier.}}, Springer, 2004.

\bibitem{Lep19}
M.~Lepoutre, \emph{Blossoming bijection for higher-genus maps},
  J.~Combin.~Theory~Ser.~A \textbf{165} (2019), 187--224.

\bibitem{LepoutrePhD}
{M}. Lepoutre, \emph{Blossoming bijections, multitriangulations: What about
  other surfaces?}, Ph.D. thesis, University of Paris-Saclay, France, 2019.

\bibitem{HDRLeveque}
B.~L{\'e}v{\^e}que, \emph{Generalization of {S}chnyder woods to orientable
  surfaces and applications}, Habilitation \`a diriger des recherches,
  Université Grenoble Alpes, 2017.

\bibitem{Propp93}
J.~Propp, \emph{Lattice structure for orientations of graphs}, arXiv preprint
  math/0209005 (2002).

\bibitem{Scha97}
G.~Schaeffer, \emph{Bijective census and random generation of {E}ulerian planar
  maps with prescribed vertex degrees}, Electron. J. Combin. \textbf{4} (1997),
  no.~1, 20.

\bibitem{Scha98}
\bysame, \emph{Conjugaison d'arbres et cartes combinatoires al{\'e}atoires},
  Ph.D. thesis, Bordeaux 1, 1998.

\bibitem{Scha15}
\bysame, \emph{{Planar Maps}}, Handbook of {E}numerative {C}ombinatorics
  (M~B{\'{o}}na, ed.), vol.~87, CRC Press, 2015.

\bibitem{Tut62}
W.T. Tutte, \emph{A census of slicings}, Canad. J. Math. \textbf{14} (1962),
  708--722.

\bibitem{Tut63}
\bysame, \emph{A census of planar maps}, Canad. J. Math. \textbf{15} (1963),
  249--271.

\bibitem{WalshLehmanI}
T.~Walsh and A.~Lehman, \emph{Counting rooted maps by genus. {I}},
  J.~Combin.~Theory~Ser.~B \textbf{13} (1972), no.~3, 192--218.

\bibitem{WalshLehmanII}
\bysame, \emph{Counting rooted maps by genus {II}}, J.~Combin.~Theory~Ser.~B
  \textbf{13} (1972), no.~2, 122--141.

\end{thebibliography}

\end{document}